\documentclass[preprint,review,12pt]{elsarticle}

 \usepackage{graphics}
\usepackage{graphicx}

\usepackage{amssymb}
\usepackage{amsthm,amsmath}
\usepackage{subfigure}
\usepackage{multirow}

\usepackage{citesort,natbib}

\biboptions{sort&compress}

\journal{Journal of Multivariate Analysis}

\begin{document}

\begin{frontmatter}


\title{ Extreme Eigenvalue Distributions of Some Complex Correlated Non-Central Wishart
and Gamma-Wishart Random Matrices}


\author{Prathapasinghe Dharmawansa\corref{mc}}
\ead{prathapakd@gmail.com}
\author{Matthew R. McKay}
\ead{eemckay@ust.hk}

\address{Department of Electronic and Computer Engineering, Hong Kong University of Science and Technology, Clear Water Bay, Kowloon, Hong Kong}
\cortext[mc]{Corresponding author. Fax: (852) 2358 1485. }

\begin{abstract}
Let $\mathbf{W}$ be a correlated complex non-central Wishart matrix defined through $\mathbf{W}=\mathbf{X}^H\mathbf{X}$, where $\mathbf{X}$ is $n\times m \, (n\geq m)$ complex Gaussian with non-zero mean $\boldsymbol{\Upsilon}$ and non-trivial covariance $\boldsymbol{\Sigma}$. We derive exact expressions for the cumulative distribution functions (c.d.f.s) of the extreme eigenvalues (i.e., maximum and minimum) of $\mathbf{W}$  for some particular cases. These results are quite simple, involving rapidly converging infinite series, and apply for the practically important case where $\boldsymbol{\Upsilon}$ has rank one. We also derive analogous results for a certain class of gamma-Wishart random matrices, for which $\boldsymbol{\Upsilon}^H\boldsymbol{\Upsilon}$ follows a matrix-variate gamma distribution. The eigenvalue distributions in this paper have various applications to wireless communication systems, and arise in
other fields such as econometrics, statistical physics, and multivariate statistics.
\end{abstract}

\end{frontmatter}

\newtheorem{theorem}{Theorem}
\newtheorem{definition}{Definition}
\newtheorem{lemma}{Lemma}
\newtheorem{corollary}{Corollary}
\newtheorem{remark}{Remark}


\section{Introduction}
\label{intro}

Eigenvalue distributions of Wishart random matrices arise in many
fields. Prominent examples include wireless communication systems
\cite{telatar,Shijin,Alouini,Matthew,MatthewIT,Chi,PeterJ}, synthetic
aperture radar (SAR) signal processing \cite{Martinez}, econometrics
\cite{Stock}, statistical physics \cite{Bronk,Wig}, and multivariate
statistical analysis \cite{James1964,Chuk,Sengupta,Guptanew}. In many cases, the Wishart
matrices of interest are complex \cite{Good}, correlated, and non-central.  Such
matrices arise, for example, in multiple-input
multiple-output (MIMO) communication channels
characterized by line-of-sight components (i.e.,\ Rician
fading) with spatial correlation amongst the antenna elements
\cite{MatthewIT}.

In this paper, a main focus is on the distributions of the \emph{extreme} eigenvalues (i.e., maximum and minimum) of Wishart matrices, which arise in many areas. For example, in the context of contemporary wireless communication systems, the maximum eigenvalue distribution is instrumental to the analysis of MIMO multi-channel beamforming systems \cite{Shijin} and the analysis of MIMO maximal ratio combining receivers \cite{Alouini,Matthew}, whereas the minimum eigenvalue distribution is important for the design and analysis of adaptive MIMO multiplexing-diversity switching systems \cite{Robert1}, as well as the analysis of linear MIMO receiver structures \cite{Nara}. In the context of econometrics, the minimum eigenvalue of a non-central Wishart matrix is important for characterizing the weak instrument asymptotic distribution of the Cragg-Donald statistic \cite{Stock}. In statistical physics, information pertaining to the nature of entanglement of a random pure quantum state can be obtained from the two extreme eigenvalue densities of Wishart matrices \cite{Satya}. Moreover, the maximal and minimal height distributions of $N$ non-intersecting fluctuating interfaces at the thermal equilibrium and with a certain external potential are also related to the extreme eigenvalues of a Wishart matrix \cite{Nadal}. As a final example, in SAR signal processing, the probability density of the maximum eigenvalue a Wishart matrix is an important parameter for target detection and analysis \cite{Martinez}.

We focus primarily on correlated complex non-central Wishart matrices, as well as another important and closely related class of random matrices, which we refer to \emph{gamma-Wishart}. Such matrices arise in the context of MIMO land mobile satellite (LMS) communication systems \cite{Alfano}, and correspond to non-central Wishart matrices with a  random non-centrality matrix having a distribution which is intimately related to the matrix-variate gamma. As discussed in \cite{Alfano}, the eigenvalues of gamma-Wishart random matrices are important for the design and analysis of MIMO LMS systems; for example, the maximum eigenvalue density determines the performance of beamforming transmission techniques, whereas the minimum eigenvalue density is closely related to the performance of linear reception techniques.

Recently, the marginal eigenvalue distributions of random matrices have received much attention; for surveys, see \cite{Verdu,Edleman,McKayThesis}. For the extreme eigenvalues, distributional results are now available for correlated central, uncorrelated central, and uncorrelated
noncentral complex Wishart matrices (see, for example, \cite{Khatri1964,Khatri1968,Matthew,Chi,Forrester,Forrester2,ChenManning,Aris,Alouini,Shijin,Ranjan1,Maraf,Zanella,Kov1,Kov2,Rathna1}). Far less is known for gamma-Wishart matrices, other than the results in \cite{Alfano}, which deal exclusively with uncorrelated matrices. In the majority of cases, the standard approach has been to integrate the respective joint eigenvalue densities over suitably chosen multi-dimensional regions. For the more general class of complex non-central Wishart and gamma-Wishart matrices with \emph{non-trivial correlation} however, there appears to be no tractable existing results. For these matrices, as we will show, the joint eigenvalue densities are extremely complicated, and it seems that this direct approach cannot be easily undertaken to yield meaningful results.

In this paper, by employing an alternative derivation technique (also considered in \cite{Davis1979,Muirhead,Const,Mathai,Rathna,Kov1}) which allows us to deal with the joint matrix-variate density rather than the density of the eigenvalues, we derive new exact expressions for the cumulative distribution
functions (c.d.f.s) of the minimum and maximum eigenvalues of correlated complex
non-central Wishart and correlated gamma-Wishart random matrices. In both cases, whilst a general theory which accounts for all matrix dimensions and distributional parameters appears intractable, we are able to derive solutions for various important scenarios.  Specifically, for correlated non-central Wishart matrices, we derive expressions for the minimum eigenvalue c.d.f.s when the matrix dimensionality and the number of degrees of freedom are equal. We also derive results for some specific scenarios for which they are not equal, and present some analogous results for the maximum eigenvalue c.d.f.  For tractability, we focus on matrices with rank-one non-centrality parameter, which is practical for various applications; most notably, MIMO communication systems with a direct line-of-sight path between the transmitter and receiver. Given the overwhelming complexity of the underlying joint eigenvalue distribution, these extreme eigenvalue c.d.f.\ expressions are remarkably simple, involving infinite series with fast convergence, and they can be easily and efficiently computed.

For the case of gamma-Wishart matrices, we focus on scenarios for which the underlying matrix-variate gamma has an integer parameter.  The implications of this assumption from a telecommunications engineering perspective are discussed in \cite{Alfano}.  As for the non-central Wishart case, we derive exact expressions for the minimum and maximum eigenvalue distributions for certain gamma-Wishart particularizations.

Whilst
previous expressions pertaining to the non-central Wishart case have been reported in \cite{Davis1979,Rathna,Mathai};
those are very complicated, involving either
infinite series' with inner summations over partitions with each term involving invariant zonal polynomials (c.f.\ Section
\ref{sec:Prelim}), or infinite series with special functions of matrix arguments \cite{Davis1979,Mathai}. As such, those previous results have limited utility from a numerical computation perspective.


\section{Preliminaries and New Matrix Integrals} \label{sec:Prelim}
\subsection{Preliminaries}
In this section, we provide some preliminary results and definitions
in random matrix theory which will be useful in the subsequent
derivations. The following notation is used throughout the paper. Matrices are
represented as uppercase bold-face, and vectors by lowercase
bold-face. The superscript $(\cdot)^H$ indicates the
Hermitian-transpose. $\mathbf{I}_p$
denotes a $p\times p$ identity matrix. We use $|\cdot|$ to represent
the determinant of a square matrix, $\text{tr}(\cdot)$ to represent
trace, and $\text{etr}(\cdot)$ stands for
$\exp\left(\text{tr}(\cdot)\right)$. The set of complex Hermitian $m\times m$ matrices are denoted by $\mathcal{H}_m$ and the set of Hermitian positive definite matrices are denoted as $\mathcal{H}_m^+$.  For $\mathbf{A},\mathbf{B}\in\mathcal{H}_m$, $\mathbf{A}>0$ is used to indicate the positive definiteness, and $\mathbf{A}>\mathbf{B}$ denotes $\mathbf{A}-\mathbf{B}\in\mathcal{H}_m^+$. $\mathbf{A}\geq 0$ is used
to indicate non-negativeness. $\mathbf{A}_{j,k}$ represents the $j,k$th element of matrix $\mathbf{A}$. $\left\lceil x\right\rceil$ is the ceiling function, defined as $\left\lceil x\right\rceil=\min\left\{n\in \mathbb{Z}|n\geq x\right\}$. Finally, the $k$th derivative of function $f(y)$ is represented as $f^{(k)}(y)$ for all $k\in\mathbb{Z}^+$, and with $f^{(0)}(y):=f(y)$.
\begin{definition}
The generalized hypergeometric function of one matrix argument can
be defined as\footnote{The convergence of the infinite zonal series
is discussed in \cite{Muirhead,Rathna}.}
\begin{equation}
\label{hypo}
{}_p\widetilde{F}_q\left(a_1,a_2,\ldots,a_p;b_1,b_2,\ldots,b_q;\mathbf{Y}\right)=\sum_{k=0}^\infty
\sum_{\kappa}
\frac{[a_1]_\kappa[a_2]_\kappa\cdots[a_p]_\kappa}{[b_1]_\kappa[b_2]_\kappa\cdots[b_q]_\kappa}
\frac{C_{\kappa}(\mathbf{Y})}{k!}
\end{equation}
where $\mathbf{Y}\in \mathcal{H}_m$, $[a]_\kappa=\displaystyle \prod_{j=1}^m(a-j+1)_{k_j}$,
$\kappa=\left(k_1,k_2,\ldots,k_m\right)$ is a partition of $k$ such that $k_1\geq k_2\geq\ldots\geq k_m\geq 0$ and $\sum_{i=1}^mk_i=k$, and $(a)_k=a(a+1)\cdots (a+k-1)$. Also, the complex zonal polynomial $C_{\kappa}(\mathbf{Y})$ is defined in \cite{James1964}.
\end{definition}
\begin{remark}
Note that the infinite zonal polynomial expansion given in (\ref{hypo}) reduces to a finite series if at least one of the $a_i$s is a negative integer. As such, when $N\in\mathbb{Z}^+$ we have
\begin{equation}
\label{hyptrk}
{}_p\widetilde{F}_q\left(-N,a_2,\ldots,a_p;b_1,b_2,\ldots,b_q;\mathbf{Y}\right)=\sum_{k=0}^{mN}
\widetilde \sum_{\kappa}
\frac{[-N]_\kappa[a_2]_\kappa\cdots[a_p]_\kappa}{[b_1]_\kappa[b_2]_\kappa\cdots[b_q]_\kappa}
\frac{C_{\kappa}(\mathbf{Y})}{k!}
\end{equation}
where $\widetilde \sum_{\kappa}$ denotes the summation over all
partitions $\kappa=\left(k_1,k_2,\ldots,k_m\right)$ of $k$ with
$k_1\leq N$.
 \end{remark}
For more properties of zonal polynomials, see
\cite{James1968,Takemura,Caro}.
\begin{definition}{\bf{(Non-Central Wishart Distribution)}}
Let $\mathbf{X}$ be an $n\times m$ ($n \geq m$) random matrix distributed as
$\mathcal{CN}_{n,m}\left(\boldsymbol{\Upsilon},\mathbf{I}_n\otimes
\boldsymbol{\Sigma }\right)$, where $\boldsymbol{\Sigma}\in
\mathcal{H}_m^+$ and $\boldsymbol{\Upsilon}\in
\mathbb{C}^{n\times m}$. Then $\mathbf{W}=\mathbf{X}^H\mathbf{X}\in\mathcal{H}_m^+$
has a complex non-central Wishart
distribution
$\mathcal{W}_m\left(n,\boldsymbol{\Sigma},\boldsymbol{\Theta}\right)$
with density function  \cite{James1964}
\begin{equation}
\label{wishart}
\begin{split}
f_{\mathbf{W}}\left(\mathbf{W}\right)& =
\frac{\mathrm{etr}\left(-\boldsymbol{\Theta}\right)|\mathbf{W}|^{n-m}}{\tilde{\Gamma}_m(n)|\boldsymbol{\Sigma|}^{n}}
\mathrm{etr}\left(-\boldsymbol{\Sigma}^{-1}\mathbf{W}\right){}_0\widetilde{F}_1\left(n;\boldsymbol{\Theta}\boldsymbol{\Sigma}^{-1}\mathbf{W}\right)
\end{split}
\end{equation}
where $\boldsymbol{\Theta}=\boldsymbol{\Sigma}^{-1}\boldsymbol{\Upsilon}^H\boldsymbol{\Upsilon}$
is the \textit{non-centrality} parameter and
$\tilde{\Gamma}_m(\cdot)$ represents the complex multivariate gamma
function defined as
\begin{equation*}
\tilde{\Gamma}_m(n)\stackrel{ \Delta}{=}\pi^{\frac{m(m-1)}{2}}\prod_{j=1}^{m}\Gamma(n-j+1)
\end{equation*}
with $\Gamma(\cdot)$ denoting the classical gamma function.
\end{definition}

\begin{definition}{\bf{(Matrix Variate Gamma Distribution)}}
Let $\alpha\geq m$ and $\boldsymbol{\Omega}\in\mathcal{H}_m^+$. The random matrix $\mathbf{M}\in\mathcal{H}_m^+$ has a matrix-variate complex gamma distribution $\Gamma_m\left(\alpha,\boldsymbol{\Omega}\right)$ if its density is \cite[Def. 6.3]{Mathai2}.
\end{definition}
\begin{definition}{\bf{(Gamma-Wishart Distribution)}}
Let us construct an $n\times m$ matrix $\widetilde{\mathbf{X}}$ such that
\begin{equation}
\widetilde{\mathbf{X}}=\widehat{\mathbf{X}}+\overline{\mathbf{X}}
\end{equation}
where $\widehat{\mathbf{X}}\sim\mathcal{CN}_{n,m}\left(\mathbf{0},\mathbf{I}_n\otimes
\boldsymbol{\Sigma }\right) $ and $\overline{\mathbf{X}}^H\overline{\mathbf{X}}\sim \Gamma_m\left(\alpha,\boldsymbol{\Omega}\right)$ are independent. Then  $\mathbf{V}=\widetilde{\mathbf{X}}^H\widetilde{\mathbf{X}}\in\mathcal{H}_m^+$ follows a gamma-Wishart distribution $\Gamma {\cal W}_m (n, \alpha, \boldsymbol{\Sigma}, \boldsymbol{\Omega})$ given by \cite{Alfano}
\begin{equation}
\label{Gram}
f_{\mathbf{V}}(\mathbf{V})=\frac{\mathrm{etr}\left(-\boldsymbol{\Sigma}^{-1}\mathbf{V}\right)|\mathbf{V}|^{n-m}|\boldsymbol{\Omega}|^{\alpha}}
{\tilde{\Gamma}_m(n)|\boldsymbol{\Sigma}|^{n}\left|\boldsymbol{\Sigma}^{-1}+\boldsymbol{\Omega}\right|^{\alpha}}
{}_1\widetilde{F}_1\left(\alpha;n;\boldsymbol{\Sigma}^{-1}\left(\boldsymbol{\Sigma}^{-1}+\boldsymbol{\Omega}\right)^{-1}
\boldsymbol{\Sigma}^{-1}\mathbf{V}\right).
\end{equation}
 Note that for $\alpha=n$, (\ref{Gram}) reduces to $\mathcal{W}_{m}\left(n,\boldsymbol{\Sigma}+\boldsymbol{\Omega}^{-1}\right)$.
\end{definition}

In addition to zonal polynomials, non-central distributional
problems in multivariate statistics commonly give rise to other
classes of invariant polynomials \cite{Chikuse1986}.

The next lemma presents the joint eigenvalue distributions of gamma-Wishart matrix, in terms of invariant polynomials defined in \cite{Davis1979,Davis1980,Rathna}. The proof of this lemma follows similar steps to the proof of the correlated non-central Wishart joint eigenvalue density, $g_{\boldsymbol{\Lambda}}\left(\boldsymbol{\Lambda}\right)$,  in \cite[Eq. 5.4]{Rathna} and thus
omitted.
\begin{lemma} \label{lem:eigwgamma}
The joint density of the ordered eigenvalues ${\lambda}_1>{\lambda}_2> \cdots >{\lambda}_m>0$, of the matrix $\mathbf{V}$ in (\ref{Gram}) is given by
\begin{align}
\label{eigenpdfwg}
g_{\widetilde{\boldsymbol{\Lambda}}}\left({\boldsymbol{\Lambda}}\right) &=
\frac{\pi^{m(m-1)}|\boldsymbol{\Omega}|^\alpha}{\tilde{\Gamma}_m(n)\tilde{\Gamma}_m(m)|\boldsymbol{\Sigma}|^{n}
|\boldsymbol{\Omega}+\boldsymbol{\Sigma}^{-1}|^\alpha}
\prod_{k=1}^{m}{\lambda}_k^{n-m}\prod_{k<l}^m\left({\lambda}_k-{\lambda}_l\right)^2\nonumber\\
& \times
\sum_{k,s=0}^{\infty}\sum_{\kappa,\sigma;\phi\in\kappa.\sigma}
\frac{[\alpha]_\sigma C_{\phi}^{\kappa,\sigma}\left(-\boldsymbol{\Sigma}^{-1},\boldsymbol{\Sigma}^{-1}\left(\boldsymbol{\Omega}+\boldsymbol{\Sigma}^{-1}\right)^{-1}
\boldsymbol{\Sigma}^{-1} \right)
C_{\phi}^{\kappa,\sigma}\left({\boldsymbol{\Lambda}},{\boldsymbol{\Lambda}}\right)}{k!s![n]_{\sigma}C_{\phi}\left(\mathbf{I}_m\right)}
\end{align}
where ${\boldsymbol{\Lambda}}$ is a diagonal matrix containing the
eigenvalues of $\mathbf{V}$ along the main diagonal.
\end{lemma}

The following technical lemma is proved in \ref{ap:A}.

\begin{lemma}\label{lem:factorize}
Let $x_1,x_2$ be the two distinct eigenvalues of $\mathbf{X}\in\mathcal{H}_2^+$.
Then, for all $n\in\mathbb{Z}^+$,
\begin{equation}
\frac{x_1^n-x_2^n}{x_1-x_2}= \sum_{i=0}^{\left\lceil
\frac{n-2}{2}\right\rceil}
(-1)^i4^ie_i^n|\mathbf{X}|^i\mathrm{tr}^{n-1-2i}(\mathbf{X})
\end{equation}
where $e_i^n$ denotes the $i$th elementary symmetric function of the
parameters
\begin{equation}
\mathcal{S}^n := \left\{\cos^2 \left(\frac{\pi}{n}\right),\cos^2
\left(\frac{2\pi}{n}\right),\ldots,\cos^2\left(
\left\lceil\frac{n-2}{2}\right\rceil\frac{\pi}{n}\right)\right\}.
\end{equation}
\end{lemma}

\subsection{New Matrix Integrals}
Here we present some new matrix integral results which will be
important in the derivations of the extreme eigenvalue
distributions, given in the following sections.

 \begin{lemma}\label{lem:1f1}
 Let $\mathbf{A}\in\mathcal{H}_2^+$ and $\mathbf{B}\in\mathcal{H}_2$ with $\mathbf{B}\geq 0$. Also,
 define  $x_1(y)$ and $x_2(y)$ as the eigenvalues of $\mathbf{A}+\mathbf{B}y$. Then, $\forall p\in\mathbb{Z}^+_0$ and
 $\Re(a)>1$,
 \begin{equation}
 \label{incomgammaint}
 \int_{\mathbf{0}}^{\mathbf{I}_2}
 |\mathbf{X}|^{a-2}
 \mathrm{etr}\left(\mathbf{AX}\right)
 \mathrm{tr}^{p}\left(\mathbf{BX}\right)d\mathbf{X}=\frac{\tilde \Gamma_2(a)\tilde \Gamma_2(2)}{\tilde \Gamma_2(a+2)}
 \phi^{(p)}_{\mathbf{A},\mathbf{B},a}(0)
 \end{equation}
 where $ \phi^{(p)}_{\mathbf{A},\mathbf{B},a}(0)$ is calculated
 recursively via
 \begin{equation}
 \label{sub1}
 \phi^{(p)}_{\mathbf{A},\mathbf{B},a}(0)=\frac{1}{h_{\mathbf{A},\mathbf{B}}(0)}
 \left(\Delta^{(p)}_{\mathbf{A},\mathbf{B},a}(0)-\sum_{j=1}^p\binom{p}{j}\phi^{(p-j)}_{\mathbf{A},\mathbf{B},a}(0)h^{(j)}_{\mathbf{A},\mathbf{B}}(0)\right)
 \end{equation}
 with initial condition
   \begin{equation}
  \label{initialcond}
  \phi^{(0)}_{\mathbf{A},\mathbf{B},a}(0)=\phi_{\mathbf{A},\mathbf{B},a}(0)=\frac{\Delta_{\mathbf{A},\mathbf{B},a}(0)}{x_1(0)-x_2(0)}\;.
  \end{equation}
Here,
 \begin{align}
 & \Delta_{\mathbf{A},\mathbf{B},a}(y)= x_1(y){}_1F_1\left(a;a+2;x_1(y)\right) {}_1F_1\left(a-1;a+1;x_2(y)\right) \nonumber \\
 & \hspace*{3cm} -  x_2(y){}_1F_1\left(a;a+2;x_2(y)\right) {}_1F_1\left(a-1;a+1;x_1(y)\right)
 \end{align}
and
  \begin{equation}
 h^{(j)}_{\mathbf{A},\mathbf{B}}(0)=x^{(j)}_1(0)-x^{(j)}_2(0),
 \end{equation}
with
 \begin{equation}
 \label{x1def}
 x^{(j)}_1(0)=\left\{\begin{array}{cl}
 \displaystyle \frac{x_1(0)\mathrm{tr}(\mathbf{B})-|\mathbf{A}|\mathrm{tr}\left(\mathbf{BA}^{-1}\right)}{x_1(0)-x_2(0)} & \mathrm{if}\; j=1\\
 \displaystyle \frac{2\left(x_1^{(1)}(0)x_2^{(1)}(0)-|\mathbf{B}|\right)}{x_1(0)-x_2(0)} & \mathrm{if}\; j=2\\
 \displaystyle \frac{\sum_{k=1}^{j-1}\binom{j}{k}x_1^{(j-k)}(0)x_2^{(k)}(0)}{x_1(0)-x_2(0)} & \mathrm{if}\; j\geq 3\;,
  \end{array}\right.
  \end{equation}
  \begin{equation}
  \label{x2def}
  x^{(j)}_2(0)=\left\{\begin{array}{cl}
 \displaystyle \frac{|\mathbf{A}|\mathrm{tr}\left(\mathbf{BA}^{-1}\right)-x_2(0)\mathrm{tr}(\mathbf{B})}{x_1(0)-x_2(0)} & \mathrm{if}\; j=1\\
 -x_1^{(j)}(0)& \mathrm{if}\; j\geq 2\; . \\
   \end{array}\right.
  \end{equation}

 \end{lemma}
\begin{proof}
See \ref{ap:B}.
\begin{lemma} \label{lem:trace}
Let $\mathbf{A}\in\mathcal{H}_m^+$ and let $\mathbf{R}\in\mathcal{H}_m$ with unit rank. Then, for $t\in \mathbb{Z}^{+}_0$ and
$\Re(a)>m-1$,
\begin{align}
\label{theq1}
\int_{\mathbf{X}\in\mathcal{H}_m^+}
\mathrm{etr}& \left(-\mathbf{A}\mathbf{X}\right)\mathrm{tr}\left(\mathbf{X}\right)|\mathbf{X}|^{a-m}\mathrm{tr}^t\left(\mathbf{R}\mathbf{X}\right)d\mathbf{X} = \nonumber\\
&(a)_t\tilde \Gamma_m(a)\mathrm{tr}^{t}\left(\mathbf{R}\mathbf{A}^{-1}\right)|\mathbf{A}|^{-a}
\left(t\; \frac{\mathrm{tr}\left(\mathbf{R}\left(\mathbf{A}^{-1}\right)^{2}\right)}{\mathrm{tr}\left(\mathbf{R}\mathbf{A}^{-1}\right)}+a\;\mathrm{tr}(\mathbf{A}^{-1})\right).
\end{align}
\end{lemma}
\begin{proof}
See \ref{ap:C}.

When the matrices are of size $2\times 2$, we can obtain the
following general result:
\begin{lemma}\label{lem:2by2tracep}
Let $\mathbf{A}\in \mathcal{H}_2^+$ and let $\mathbf{R}\in\mathcal{H}_2$ with unit rank. Then, for $p$,
$t\in\mathbb{Z}^{+}_0$ and $\Re(a)>1$,
\begin{align}
\int_{\mathbf{X}\in\mathcal{H}_2^+}
\mathrm{etr}& \left(-\mathbf{A}\mathbf{X}\right)\mathrm{tr}^p\left(\mathbf{X}\right)
\left|\mathbf{X}\right|^{a-2}\mathrm{tr}^t\left(\mathbf{R}\mathbf{X}\right)d\mathbf{X} = \nonumber\\
&p!\frac{(a)_t\tilde \Gamma_2(a)}{|\mathbf{A}|^{a+\frac{p}{2}}}\sum_{k=0}^{\min(p, t)}\frac{(-1)^k\binom{t}{k}}{|\mathbf{A}|^{\frac{k}{2}}}\mathrm{tr}^{t-k}\left(\mathbf{R}\mathbf{A}^{-1}\right)\mathrm{tr}^{k}\left(\mathbf{R}\right)
\mathcal{C}^{a+t}_{p-k}\left(\frac{\mathrm{tr}\left(\mathbf{A}\right)}{2\sqrt{\left|\mathbf{A}\right|}}\right)
\end{align}
where $\mathcal{C}_n^\nu(\cdot)$ denotes an ultraspherical
(Gegenbauer) polynomial.
\end{lemma}
\begin{proof}
See \ref{ap:D}.

\begin{lemma}\label{lem:3by4rank1}
Let $\mathbf{A}\in\mathcal{H}_3^+$ and let $\mathbf{R}(\geq 0)\in\mathcal{H}_3^+$ with unit rank. Then, for $t\in
\mathbb{Z}^{+}_0$,
\begin{align}
\label{110thoerem}
\int_{\mathbf{X}\in\mathcal{H}_3^+}
\mathrm{etr}& \left(-\mathbf{A}\mathbf{X}\right)
\mathrm{tr}^t(\mathbf{RX})C_{1,1,0}(\mathbf{X})d\mathbf{X} =\nonumber\\
&
\tilde \Gamma_3(4)|\mathbf{A}|^{-4}\left(
(4)_t \mathrm{tr}^t\left(\mathbf{R}\mathbf{A}^{-1}\right)\mathrm{tr}(\mathbf{A})+t(4)_{t-1}
\mathrm{tr}^{t-1}\left(\mathbf{R}\mathbf{A}^{-1}\right)\mathrm{tr}(\mathbf{R})\right).
\end{align}
\end{lemma}
\begin{proof}
See \ref{ap:E}.
\begin{lemma}\label{lem:tracegamma}
Let $\mathbf{A},\mathbf{B}\in\mathcal{H}_2^+$. Then, for $p,t \in \mathbb{Z}^+_0$ and $\Re(a)>1$, we have
\begin{align}
\label{twotracetheo}
\int_{\mathbf{X}\in\mathcal{H}_2^+}
& \mathrm{etr}\left(-\mathbf{A}\mathbf{X}\right)\mathrm{tr}^p\left(\mathbf{BX}\right)  \mathrm{tr}^t\left(\mathbf{X}\right)\left|\mathbf{X}\right|^{a-2}d\mathbf{X}\nonumber\\
 &\quad =p!t!|\mathbf{A}|^{-a}\tilde \Gamma_2(a)\displaystyle
 \sum_{t_1=\left\lceil\frac{t}{2}\right\rceil}^t
 \frac{(a)_{t_1}(a)_{t-t_1}\left(2t_1+1-t\right)}{\left(t_1+1\right)!\left(t-t_1\right)!}
 \displaystyle \sum_{i=0}^{\left\lceil\frac{2t_1-t-1}{2}\right\rceil}\mathcal{B}_{\tau,p,i}
\end{align}
where
\begin{align*}
\mathcal{B}_{\tau,p,i}=\sum_{k=0}^{\min(p,\;\varepsilon_{t_1,i})}
(-1)^{k+i}4^i e_i^{\tau}  \binom{\varepsilon_{t_1,i}}{k}
\mathrm{tr}^{\varepsilon_{t_1,i}-k}\left(\mathbf{A}\right)
& \mathrm{tr}^{k}\left(\mathbf{B}\right)
|\mathbf{A}|^{-\varepsilon_{t_1}-\frac{p-k}{2}}
|\mathbf{B}|^{\frac{p-k}{2}}
\\
& \quad \times
\mathcal{C}^{\varepsilon_{t_1}+a}_{p-k}\left(\frac{\mathrm{tr}
\left(\mathbf{A}^{-1}\mathbf{B}\right)}{2\sqrt{\left|\mathbf{A}^{-1}\mathbf{B}\right|}}\right),
\end{align*}
$\varepsilon_{t_1,i}=2t_1-t-2i,\;\varepsilon_{t_1}=t_1-i$, and
$\tau=\left(t_1,t-t_1\right)$ is a partition of $t$ such that
$\left\lceil\frac{t}{2}\right\rceil\leq t_1\leq t$. Moreover,
$e_i^\tau$ denotes the $i$th elementary symmetric function of the
parameters
\begin{align}
\mathcal{S}^{\tau}& :=
\left\{\cos^2\left(\frac{\pi}{2t_1-t+1}\right),\cos^2\left(\frac{2\pi}{2t_1-t+1}\right),\ldots\ldots\right.\nonumber\\
&\hspace{4.5cm}\quad
\left.\ldots,\cos^2\left(\left\lceil\frac{2t_1-t-1}{2}\right\rceil\frac{\pi}{2t_1-t+1}\right)\right\}.
\end{align}
\end{lemma}
\begin{proof}
See \ref{ap:F}.

Armed with the new results in this section, we are now in a position
to derive the extreme eigenvalue distributions of both correlated
complex non-central Wishart and gamma-Wishart matrices. These key
results are the focus of the following two sections.

\section{New Minimum Eigenvalue Distributions}

In this section, we consider the minimum eigenvalue distribution. To
evaluate this, the most direct approach is to integrate the joint
eigenvalue probability density function (p.d.f.) as follows:
\begin{align}
F_{min} (x) &= 1 - P( \lambda_1 > \cdots > \lambda_m > x ) \nonumber \\
 &= 1 - \int_{\mathcal{D}} g(\boldsymbol{\Lambda}) d \lambda_1 \cdots d \lambda_m
\end{align}
where $\mathcal{D} = \{ x < \lambda_m < \cdots < \lambda_1 \}$ and
$g(\boldsymbol{\Lambda})\in\left\{g_{\boldsymbol{\Lambda}}(\boldsymbol{\Lambda}),
g_{ \boldsymbol{\widetilde \Lambda}}(\boldsymbol{\Lambda})
\right\}$. This direct approach, however, is difficult for two main
reasons: (i) due to the presence of the invariant polynomials in
the joint eigenvalue densities, and (ii) due  the unbounded
upper limit of the integrals which makes term-by-term integration
intractable. To circumvent these complexities, in the following we
adopt an alternative derivation approach based on integrating
directly over the matrix-variate distribution itself, rather than
the distribution of the eigenvalues.

To highlight the approach, consider $\mathbf{Y}\in\mathcal{H}_m^+$ with minimum eigenvalue
$\lambda_{\text{min}}(\mathbf{Y})$ having c.d.f.\
\begin{equation}
\label{cdf}
F_{{min}}(x)=P\left(\lambda_{{min}}(\mathbf{Y})\leq
x\right)=1-P\left(\lambda_{{min}}(\mathbf{Y})>x\right) \; .
\end{equation}
The key idea is to invoke the obvious relation\footnote{This
relation has also been employed previously in
\cite{Davis1979,Muirhead,Const,Mathai,Rathna,Kov1}.}
\begin{equation} \label{eq:minEVRelation}
P\left(\lambda_{{min}}(\mathbf{Y})>x\right) = P\left(\mathbf{Y}>
x\mathbf{I}_m \right)
\end{equation}
which allows one to deal purely with the distribution of
$\mathbf{Y}$, rather than the distribution of its eigenvalues.

\subsection{Correlated Non-Central Wishart Matrices}

For the non-central Wishart scenario, we deal with the matrix
$\mathbf{W}$ with joint density given in (\ref{wishart}). Thus, with
(\ref{eq:minEVRelation}), we have
\begin{align}
\label{matrixcdf} P\left(\lambda_{{min}}(\mathbf{W})>x\right) &=
\int_{\mathbf{W}>x\mathbf{I}_m
}f_{\mathbf{W}}\left(\mathbf{W}\right)d\mathbf{W} \nonumber\\
&=
\frac{\exp\left(-\eta\right)}{\tilde{\Gamma}_m(n)\left|\boldsymbol{\Sigma}\right|^n}
\int_{\mathbf{W}-x\mathbf{I}_m\in\mathcal{H}_m^+}
\left|\mathbf{W}\right|^{n-m}  \text{etr}\left(-\boldsymbol{\Sigma}^{-1}\mathbf{W}\right)\nonumber\\
& \hspace{5cm} \times
{}_0\widetilde{F}_1\left(n;\boldsymbol{\Theta}\boldsymbol{\Sigma}^{-1}\mathbf{W}\right)d\mathbf{W}
\end{align}
where $\eta=\text{tr}(\boldsymbol{\Theta})$. Applying the change of
variables $\mathbf{W}=x\left(\mathbf{I}_m+\mathbf{Y}\right)$ with
$d\mathbf{W}=x^{m^2}d\mathbf{Y}$ yields
\begin{align*}
P\left(\lambda_{{min}}(\mathbf{W})>x\right)&=
\frac{x^{mn}\exp\left(-\eta\right)\text{etr}\left(-x\boldsymbol{\Sigma}^{-1}\right)}{\tilde{\Gamma}_m(n)\left|\boldsymbol{\Sigma}\right|^n}
\int_{\mathbf{Y}\in\mathcal{H}_m^+}
\left|\mathbf{I}_m+\mathbf{Y}\right|^{n-m}\nonumber\\
&  \qquad \quad \times \text{etr}\left(-x\boldsymbol{\Sigma}^{-1}\mathbf{Y}\right)
{}_0\widetilde{F}_1\left(n;x\boldsymbol{\Theta}\boldsymbol{\Sigma}^{-1}\left(\mathbf{I}_m+\mathbf{Y}\right)\right)d\mathbf{Y}.
\end{align*}
It is convenient to now expand the hypergeometric function with its
equivalent zonal polynomial series expansion (\ref{hypo}) to give
\begin{align}
\label{zonal}
& P\left(\lambda_{{min}}(\mathbf{W})>x\right)=
\frac{x^{mn}\exp\left(-\eta\right)\text{etr}\left(-x\boldsymbol{\Sigma}^{-1}\right)}{\tilde{\Gamma}_m(n)\left|\boldsymbol{\Sigma}\right|^n}
\sum_{k=0}^\infty\sum_{\kappa}\frac{1}{k![n]_{\kappa}}
\nonumber\\
& \qquad\times \int_{\mathbf{Y}\in\mathcal{H}_m^+}
\left|\mathbf{I}_m+\mathbf{Y}\right|^{n-m}
\text{etr}\left(-x\boldsymbol{\Sigma}^{-1}\mathbf{Y}\right)
C_{\kappa}\left(x\boldsymbol{\Theta}\boldsymbol{\Sigma}^{-1}\left(\mathbf{I}_m+\mathbf{Y}\right)\right)d\mathbf{Y}
\end{align}
where $\kappa=\left(\kappa_1,....,\kappa_m\right)$ is a partition of $k$ into not more than $m$ parts such that $\kappa_1\geq....\geq\kappa_m\geq 0$ and $\sum_{i}^m\kappa_i=k$.

Observing that $\boldsymbol{\Theta}\boldsymbol{\Sigma}^{-1}$ is
Hermitian non-negative definite with rank one, it can be represented
via its eigen decomposition as
\begin{equation}
\label{eigendecom}
\boldsymbol{\Theta}\boldsymbol{\Sigma}^{-1}=\mu \boldsymbol{\alpha}\boldsymbol{\alpha}^H
\end{equation}
where $\boldsymbol{\alpha}\in\mathbb{C}^{m\times 1}$ and
$\boldsymbol{\alpha}^H\boldsymbol{\alpha}=1$. Recalling that zonal
polynomials depend only on the \emph{eigenvalues} of their matrix
arguments, and noting that
$\boldsymbol{\Theta}\boldsymbol{\Sigma}^{-1}\left(\mathbf{I}_m+\mathbf{Y}\right)$
is also rank one, we can write (\ref{zonal}) with the aid of
(\ref{eigendecom}) as
\begin{align}
\label{zonaltrace}
& P\left(\lambda_{{min}}(\mathbf{W})>x\right)=
\frac{x^{mn}\exp\left(-\eta\right)\text{etr}\left(-x\boldsymbol{\Sigma}^{-1}\right)}{\tilde{\Gamma}_m(n)\left|\boldsymbol{\Sigma}\right|^n}
\sum_{k=0}^\infty\sum_{\kappa}\frac{1}{k![n]_{\kappa}}
\nonumber\\
&\quad\times \int_{\mathbf{Y}\in\mathcal{H}_m^+}
\left|\mathbf{I}_m+\mathbf{Y}\right|^{n-m}
\text{etr}\left(-x\boldsymbol{\Sigma}^{-1}\mathbf{Y}\right)
C_{\kappa}\left(x\mu\boldsymbol{\alpha}^H\left(\mathbf{I}_m+\mathbf{Y}\right)\boldsymbol{\alpha}\right)d\mathbf{Y}.
\end{align}
Applying the complex analogue of \cite[Corollary 7.2.4]{Muirhead},
since
$\boldsymbol{\alpha}^H\left(\mathbf{I}_m+\mathbf{Y}\right)\boldsymbol{\alpha}$
is rank one, then it follows that $C_{\kappa}\left(x\mu\boldsymbol{\alpha}^H\left(\mathbf{I}_m+\mathbf{Y}\right)\boldsymbol{\alpha}\right)=0$
for all partitions $\kappa$ having more than one non-zero part.
Hence
\begin{align}
\label{zonaldef}
C_{\kappa}\left(x\mu\boldsymbol{\alpha}^H\left(\mathbf{I}_m+\mathbf{Y}\right)\boldsymbol{\alpha}\right)
&= (x \mu)^k \sum_{t=0}^{k}\binom{k}{t}\text{tr}^t\left(\boldsymbol{\alpha}\boldsymbol{\alpha}^H\mathbf{Y}\right)
\end{align}
and (\ref{zonaltrace}) can be written as
\begin{align}
\label{cdfintegral}
P\left(\lambda_{{min}}(\mathbf{W})>x\right)=
\frac{x^{mn}\exp\left(-\eta\right)\text{etr}\left(-x\boldsymbol{\Sigma}^{-1}\right)}{\tilde{\Gamma}_m(n)\left|\boldsymbol{\Sigma}\right|^n}
\sum_{k=0}^\infty\frac{\left(x\mu\right)^k}{k!(n)_{k}}\sum_{t=0}^k\binom{k}{t}\mathcal{Q}^t_{m,n}(x)
\end{align}
where
\begin{equation}
\label{finalmatintegra}
\mathcal{Q}^t_{m,n}(x)=\int_{\mathbf{Y}\in\mathcal{H}_m^+}
\left|\mathbf{I}_m+\mathbf{Y}\right|^{n-m}
\text{etr}\left(-x\boldsymbol{\Sigma}^{-1}\mathbf{Y}\right)
\text{tr}^t\left(\boldsymbol{\alpha}\boldsymbol{\alpha}^H\mathbf{Y}\right) d\mathbf{Y}.
\end{equation}
Unfortunately, it appears that this integral is not solvable in
closed form for \emph{arbitrary} values of $m$ and $n$. However, as
we now show, it can be solved in closed-form for various important
configurations, thus yielding exact expressions for the minimum
eigenvalue distributions. These results are presented in three key
theorems.  In each of these, we recall the notation
\begin{align}
\mu = {\rm tr} \left( \boldsymbol{\Theta}\boldsymbol{\Sigma}^{-1}
\right) , \quad \quad \eta = {\rm tr} \left( \boldsymbol{\Theta}
\right) .
\end{align}

The theorem below gives the exact minimum eigenvalue distribution
for ``square'' Wishart matrices:

\begin{theorem} \label{th:MainResult}
Let $\mathbf{X}\sim \mathcal{CN}_{m,m}\left(\boldsymbol{\Upsilon},\mathbf{I}_m\otimes\boldsymbol{\Sigma}\right)$, where $\boldsymbol{\Upsilon}\in\mathbb{C}^{m\times m}$ has rank one, and $\mathbf{W}=\mathbf{X}^H\mathbf{X}$. Then the c.d.f.\ of $\lambda_{\text{min}}(\mathbf{W})$ is given by
\begin{equation}
\label{cdfans}
F_{{\text{min}}}(x)=
1-\exp\left(-\eta\right)\mathrm{etr}\left(-x\boldsymbol{\Sigma}^{-1}\right)
\sum_{k=0}^{\infty}
\frac{\left(x\mu\right)^k}{k!(m)_{k}}
{}_1F_1\left(m;m+k;\eta\right).
\end{equation}
\end{theorem}
\begin{proof}
Substituting $m=n$ into (\ref{cdfintegral}) and
(\ref{finalmatintegra}) yields
\begin{align}
\label{cdfintmm}
P\left(\lambda_{{min}}(\mathbf{W})>x\right)=&
\frac{x^{m^2}\exp\left(-\eta\right)\text{etr}\left(-x\boldsymbol{\Sigma}^{-1}\right)}{\tilde{\Gamma}_m(m)\left|\boldsymbol{\Sigma}\right|^m}
\sum_{k=0}^\infty\frac{\left(x\mu\right)^k}{k!(m)_{k}}\sum_{t=0}^k\binom{k}{t}\mathcal{Q}^t_{m,m}(x)
\end{align}
where
\begin{equation}
\mathcal{Q}^t_{m,m}(x)=\int_{\mathbf{Y}\in\mathcal{H}_m^+}
\text{etr}\left(-x\boldsymbol{\Sigma}^{-1}\mathbf{Y}\right)
C_{\tau}\left(\boldsymbol{\alpha}\boldsymbol{\alpha}^H\mathbf{Y}\right) d\mathbf{Y}.
\end{equation}
This matrix integral can be solved using \cite[Eq. 6.1.20]{Mathai2} to give
\begin{equation}
\label{intfinalsol} \mathcal{Q}^t_{m,m}(x)=\frac{\tilde{\Gamma}_m(m)(m)_t
\left|\boldsymbol{\Sigma}\right|^m}{x^{m^2}}C_{\tau}\left(\frac{\boldsymbol{\Theta}}{\mu
x}\right)=\frac{\tilde{\Gamma}_m(m)(m)_t\left|\boldsymbol{\Sigma}\right|^m}{x^{m^2}}\left(\frac{\eta}{x\mu}\right)^t \; \end{equation}
where we have applied (\ref{eigendecom}) to arrive at the argument
of the zonal polynomial.
Substituting (\ref{intfinalsol}) into (\ref{cdfintmm}) with some
manipulation yields
\begin{equation}
\label{befresum}
P\left(\lambda_{{min}}(\mathbf{W})>x\right)=
\exp\left(-\eta\right)\text{etr}\left(-x\boldsymbol{\Sigma}^{-1}\right)
\sum_{k=0}^{\infty}
\frac{\left(x\mu\right)^k}{k!(m)_{k}}
\sum_{t=0}^k\binom{k}{t}(m)_t\left(\frac{\eta}{x\mu}\right)^t.
\end{equation}
To obtain a power series in $x$, we re-sum the infinite series as
follows
\begin{align}
\label{resum}
\sum_{k=0}^{\infty}
\frac{\left(x\mu\right)^k}{k!(m)_{k}}
\sum_{t=0}^k\binom{k}{t}(m)_t\left(\frac{\eta}{x\mu}\right)^t=
\sum_{k=0}^\infty\frac{\left(x\mu\right)^{k}}{k!(m)_k}{}_1F_1\left(m;m+k;\eta\right).
\end{align}
Finally, using (\ref{resum}) in (\ref{befresum}) with (\ref{cdf})
gives the result in (\ref{cdfans}).
\end{proof}
\begin{remark}
An alternative expression for the c.d.f. can be obtained by observing the fact that
\begin{align}
\sum_{k=0}^{\infty}
\frac{\left(x\mu\right)^k}{k!(m)_{k}}
\sum_{t=0}^k\binom{k}{t}(m)_t\left(\frac{\eta}{x\mu}\right)^t&=\sum_{t=0}^\infty\sum_{k=0}^\infty
\frac{(m)_t}{(m)_{t+k}t!k!}\eta^t\left(x\mu\right)^k\nonumber\\
& = \Phi_3\left(m,m,\eta,x\mu\right)
\end{align}
where $\Phi_3(a,b,x,y)$ is the confluent hypergeometric function of two variables \cite[Eq. 5.7.1.23 ]{Erdelyi}. Thus, we can write the minimum eigenvalue c.d.f. as
\begin{equation}
F_{{min}}(x)=1-\exp\left(-\eta\right)\mathrm{etr}\left(-x\boldsymbol{\Sigma}^{-1}\right)\Phi_3\left(m,m,\eta,x\mu\right).
\end{equation}
\end{remark}

The theorem below gives the exact minimum eigenvalue distribution
for $2 \times 2$ Wishart matrices with \emph{arbitrary} degrees of
freedom:
\begin{theorem}\label{th:nby2wishart}
Let $\mathbf{X}\sim \mathcal{CN}_{n,2}\left(\boldsymbol{\Upsilon},\mathbf{I}_n\otimes\boldsymbol{\Sigma}\right)$, where $\boldsymbol{\Upsilon}\in\mathbb{C}^{n\times 2}$ has rank one, and $\mathbf{W}=\mathbf{X}^H\mathbf{X}$. Then the c.d.f.\ of $\lambda_{\text{min}}(\mathbf{W})$ is given by
\begin{align}
\label{cdfn2}
F_{{\text{min}}}(x)=
1-\exp\left(-\eta\right)\frac{\mathrm{etr}\left(-x\boldsymbol{\Sigma}^{-1}\right)}{\tilde \Gamma_2(n)\left|\boldsymbol{\Sigma}\right|^{n-2}}
\sum_{k=0}^\infty
\frac{\left(x\mu\right)^k}{k!(n)_k}
\sum_{t=0}^k
\binom{k}{t}\left(\frac{\eta}{x\mu}\right)^t\rho(t,x)
\end{align}
where
\begin{align*}
\rho(t,x)=\sum_{i=0}^{n-2}\sum_{j=0}^i\sum_{l=0}^{\min(j,t)}
&(-1)^l\binom{n-2}{i}\binom{i}{j}\binom{t}{l}  j!(\omega_{i,j})_t\tilde \Gamma_2\left(\omega_{i,j}\right)\left(\frac{\mu}{\eta}\right)^l
\\
& \times
\left|\boldsymbol{\Sigma}\right|^{i+l/2-j/2}\mathcal{C}_{j-l}^{\omega_{i,j}+t}\left(\frac{1}{2}\mathrm{tr}
\left(\boldsymbol{\Sigma}^{-1}\right)\sqrt{\left|\boldsymbol{\Sigma}\right|}\right)x^{2n+j-2i-4}\;,
\end{align*}
and $\omega_{i,j}=i-j+2$.
\end{theorem}
\begin{proof}
We begin by substituting $m=2$ into (\ref{cdfintegral}) and
(\ref{finalmatintegra}) to yield
\begin{align*}
\label{ncdfexp}
P\left(\lambda_{min}(\mathbf{W})>x\right)=\frac{\exp(-\eta)}{\tilde \Gamma_2(n)\left|\boldsymbol{\Sigma}\right|^n}x^{2n}\text{etr}\left(-x\boldsymbol{\Sigma}^{-1}\right)\sum_{k=0}^\infty
\frac{\left(x\mu\right)^k}{k!(n)_k}\sum_{t=0}^k \binom{k}{t}\mathcal{Q}^t_{2,n}(x).
\end{align*}
Now we use the determinant expansion
\begin{equation}
\label{ncdfdetexp}
\left|\mathbf{I}_2+\mathbf{Y}\right|^{n-2}=\sum_{i=0}^{n-2}\sum_{j=1}^i
\binom{n-2}{i}\binom{i}{j}
\text{tr}^j\left(\mathbf{Y}\right)\left|\mathbf{Y}\right|^{i-j}
\end{equation}
to write $\mathcal{Q}^t_{2,n}(x)$ as
\begin{align}
\mathcal{Q}^t_{2,n}(x)=\sum_{i=0}^{n-2}\sum_{j=1}^i
\binom{n-2}{i}\binom{i}{j}
\int_{\mathbf{Y}\in\mathcal{H}_2^+}
\text{tr}^j\left(\mathbf{Y}\right)\left|\mathbf{Y}\right|^{i-j}
& \text{etr}\left(-x\boldsymbol{\Sigma}^{-1}\mathbf{Y}\right)\nonumber\\
& \hspace*{-1cm} \times
\mathrm{tr}^t\left(\boldsymbol{\alpha}\boldsymbol{\alpha}^H\mathbf{Y}\right)
d\mathbf{Y}.
\end{align}
 Lemma \ref{lem:2by2tracep} can be used to solve the above integral in closed form and subsequent use of (\ref{cdf}) followed by some algebraic manipulations gives (\ref{cdfn2}).
\end{proof}

Although the c.d.f.\ result in Theorem \ref{th:nby2wishart} is
seemingly complicated, it can be evaluated numerically for any value
of $n$.  Moreover, for specific values of $n$ it often gives
simplified solutions.  Some examples are shown in the following
corollaries.
\begin{corollary}\label{cor:3by2wishart}
Let $\mathbf{X}\sim \mathcal{CN}_{3,2}\left(\boldsymbol{\Upsilon},\mathbf{I}_3\otimes\boldsymbol{\Sigma}\right)$, where $\boldsymbol{\Upsilon}\in\mathbb{C}^{3\times 2}$ has rank one, and $\mathbf{W}=\mathbf{X}^H\mathbf{X}$. Then the c.d.f.\ of  $\lambda_{\text{min}}(\mathbf{W})$ is given by
\begin{equation}
\begin{split}
\label{cdfans3}
F_{{\text{min}}}(x)=
1-\exp\left(-\eta\right)\mathrm{etr}\left(-x\boldsymbol{\Sigma}^{-1}\right)
& \sum_{k=0}^\infty\frac{\left(x\mu\right)^k}{k!(3)_k}\mathcal{F}_{3,2}(k,\eta,x)
\end{split}
\end{equation}
where
\begin{equation*}
\mathcal{F}_{3,2}(k,\eta,x)=\varrho_1(x){}_1F_1\left(3;3+k;\eta\right)+\varrho_2(x){}_1F_1\left(2;3+k;\eta\right),
\end{equation*}
\begin{align*}
\varrho_1(x)=1+\left(\mathrm{tr}\left(\boldsymbol{\Sigma}^{-1}\right)-\frac{\mu}{\eta}\right)x,\;\;\text{and}\;\;
\varrho_2(x)=\frac{\mu}{\eta}x+\frac{x^2}{2|\boldsymbol{\Sigma}|}\;.
\end{align*}
\end{corollary}
\begin{remark}
An alternative expression for the above c.d.f. can be written based
on the confluent hypergeometric function of two arguments as
\begin{align*}
F_{{\text{min}}}(x)=
1-\exp\left(-\eta\right)\mathrm{etr}\left(-x\boldsymbol{\Sigma}^{-1}\right)
& \left(\varrho_1(x)\Phi_3\left(3,3,\eta,x\mu\right)\right.\\
& \qquad \qquad \left. +\;
\varrho_2(x)\Phi_3\left(2,3,\eta,x\mu\right)\right).
\end{align*}
\end{remark}
\begin{corollary}\label{cor:4by2wishart}
Let $\mathbf{X}\sim \mathcal{CN}_{4,2}\left(\boldsymbol{\Upsilon},\mathbf{I}_4\otimes\boldsymbol{\Sigma}\right)$, where $\boldsymbol{\Upsilon}\in\mathbb{C}^{4\times 2}$ has rank one, and $\mathbf{W}=\mathbf{X}^H\mathbf{X}$. Then the c.d.f.\ of $\lambda_{\text{min}}(\mathbf{W})$ is given by
\begin{equation}
\begin{split}
\label{cdfans4}
F_{{\text{min}}}(x)=
1-& \exp\left(-\eta\right)\mathrm{etr}\left(-x\boldsymbol{\Sigma}^{-1}\right)
\sum_{k=0}^\infty\frac{\left(x\mu\right)^k}{k!(4)_k}
\mathcal{F}_{4,2}(k,\eta,x)
\end{split}
\end{equation}
where
\begin{align*}
\mathcal{F}_{4,2}(k,\eta,x)=
\nu_1 (x){}_1F_1\left(4;4+k;\eta\right)+\nu_2 (x)&{}_1F_1\left(3;4+k;\eta\right)\\
& +
\nu_3 (x)
{}_1F_1\left(2;4+k;\eta\right),
\end{align*}
\begin{align*}
\nu_1 (x)=& 1+a_1x+\frac{a_1}{2}x^2,\\
\nu_2 (x)= & \frac{\mu}{\eta}x+\left(\frac{1}{3}+\frac{a_2}{3}+\frac{2}{3}\mathrm{tr}\left(\boldsymbol{\Sigma}^{-1}\right)a_1-a_1^2\right)x^2
+\frac{a_1}{3|\boldsymbol\Sigma|}x^3,\\
 \nu_3 (x)=& \left(\frac{a_1^2}{2}-\frac{2}{3}a_1\mathrm{tr}\left(\boldsymbol{\Sigma}^{-1}\right)-\frac{a_2}{3}+
 \frac{\mathrm{tr}^2\left(\boldsymbol{\Sigma}^{-1}\right)}{3}+\frac{\mathrm{tr}\left(\boldsymbol{\Sigma}^{-2}\right)}{6}\right)x^2\\
 & \hspace{7.5cm}+\frac{\mu x^3}{3\eta |\boldsymbol{\Sigma}|}+\frac{x^4}{12|\boldsymbol{\Sigma}|^2},
\end{align*}
$a_1=\mathrm{tr}\left(\boldsymbol{\Sigma}^{-1}\right)-\frac{\mu}{\eta}$, and $a_2=\mathrm{tr}^2\left(\boldsymbol{\Sigma}^{-1}\right)-\frac{2}{|\boldsymbol{\Sigma}|}-\frac{\mu}{\eta}$.
\end{corollary}
\begin{remark}
An alternative expression for the above c.d.f. can be written as
\begin{align*}
F_{{\text{min}}}(x)=
1-& \exp\left(-\eta\right)\mathrm{etr}\left(-x\boldsymbol{\Sigma}^{-1}\right)
\left(\nu_1 (x)\Phi_3(4,4,\eta,x \mu)
\right.\\
& \qquad \qquad \left.+ \;\nu_2 (x) \Phi_3(3,4,\eta,x \mu)+\nu_3 (x) \Phi_3(2,4,\eta,x \mu)\right).
\end{align*}
\end{remark}

The theorem below gives the exact minimum eigenvalue distribution
for $3 \times 3$ Wishart matrices with $4$ degrees of freedom:
\begin{theorem}\label{th:4by3wishart}
Let $\mathbf{X}\sim \mathcal{CN}_{4,3}\left(\boldsymbol{\Upsilon},\mathbf{I}_4\otimes\boldsymbol{\Sigma}\right)$, where $\boldsymbol{\Upsilon}\in\mathbb{C}^{4\times 3}$ has rank one, and $\mathbf{W}=\mathbf{X}^H\mathbf{X}$. Then the c.d.f.\ of $\lambda_{\text{min}}(\mathbf{W})$ is given by
\begin{align}
\label{34}
F_{{\text{min}}}(x)=
1-\exp\left(-\eta\right)\mathrm{etr}\left(-x\boldsymbol{\Sigma}^{-1}\right)
\sum_{k=0}^\infty
\frac{\left(x\mu\right)^k}{k!(4)_k}\mathcal{F}_{4,3}(k,\eta,x)
\end{align}
where
\begin{equation*}
\mathcal{F}_{4,3}(k,\eta,x)=\rho_1(x){}_1F_1\left(4;4+k;\eta \right)+\rho_2(x){}_1F_1\left(3;4+k;\eta \right),
\end{equation*}
\begin{align*}
\rho_1(x)& =
1+\left(\mathrm{tr}\left(\boldsymbol{\Sigma}^{-1}\right)-\frac{\mu}{\eta}\right)x+
\frac{\mathrm{tr}\left(\boldsymbol{\Theta\Sigma}\right)}{2\eta|\boldsymbol{\Sigma}|}x^2,\\
\rho_2(x)& =
\frac{\mu}{\eta}x+\frac{1}{2|\boldsymbol{\Sigma}|}\left(\mathrm{tr}\left(\boldsymbol{\Sigma}\right)-\mathrm{tr}\left(\boldsymbol{\Theta\Sigma}\right)
\frac{1}{\eta}\right)x^2+\frac{x^3}{6|\boldsymbol{\Sigma}|}.
\end{align*}
\end{theorem}
\begin{proof}
We can write (\ref{cdfintegral}) and (\ref{finalmatintegra}) in the case of $m=3$ and $n=4$ as
\begin{align}
\label{34cdfexp}
P\left(\lambda_{min}(\mathbf{W})>x\right)=\frac{\exp(-\eta)}{\tilde \Gamma_3(4)\left|\boldsymbol{\Sigma}\right|^4}x^{12}\text{etr}\left(-x\boldsymbol{\Sigma}^{-1}\right)\sum_{k=0}^\infty
\frac{\left(x\mu\right)^k}{k!(4)_k}\sum_{t=0}^k \binom{k}{t}\mathcal{Q}^t_{3,4}(x).
\end{align}
Following the identity
\begin{equation}
\label{det3ident}
\left|\mathbf{I}_3+\mathbf{Y}\right|=1+\text{tr}(\mathbf{Y})+|\mathbf{Y}|+C_{1,1,0}(\mathbf{Y}),
\end{equation}
 we can write $\mathcal{Q}^t_{3,4}(x)$ as
\begin{align}
\mathcal{Q}^t_{3,4}(x)=\int_{\mathbf{Y}\in\mathcal{H}_3^+}
&\text{etr}\left(-x\boldsymbol{\Sigma}^{-1}\mathbf{Y}\right)\text{tr}^t\left(\boldsymbol{\alpha}\boldsymbol{\alpha}^H\mathbf{Y}\right) d\mathbf{Y}\nonumber\\
& +\int_{\mathbf{Y}\in\mathcal{H}_3^+}
\text{etr}\left(-x\boldsymbol{\Sigma}^{-1}\mathbf{Y}\right)\text{tr}(\mathbf{Y})\text{tr}^t\left(\boldsymbol{\alpha}\boldsymbol{\alpha}^H\mathbf{Y}\right) d\mathbf{Y}\nonumber\\
& +
\int_{\mathbf{Y}\in\mathcal{H}_3^+}
\text{etr}\left(-x\boldsymbol{\Sigma}^{-1}\mathbf{Y}\right)|\mathbf{Y}|\text{tr}^t\left(\boldsymbol{\alpha}\boldsymbol{\alpha}^H\mathbf{Y}\right) d\mathbf{Y}\nonumber\\
&
+
\int_{\mathbf{Y}\in\mathcal{H}_3^+}
\text{etr}\left(-x\boldsymbol{\Sigma}^{-1}\mathbf{Y}\right)C_{1,1,0}(\mathbf{Y})\text{tr}^t\left(\boldsymbol{\alpha}\boldsymbol{\alpha}^H\mathbf{Y}\right) d\mathbf{Y}.
\end{align}
These matrix integrals can be solved with the aid of \cite[Eq.
6.1.20]{Mathai2}, Lemma \ref{lem:trace}, and Lemma
\ref{lem:3by4rank1} to yield
\begin{align}
\label{34Qans}
\mathcal{Q}^t_{3,4}(x)=\frac{|\boldsymbol{\Sigma}|^4\tilde
\Gamma_3(4)}{x^{12}}\left(\frac{\eta}{x\mu}\right)^t
\left(\rho_1(x)(4)_t+\rho_2(x) (3)_t\right)
\end{align}
where we have used the relations $t(3)_t=3(4)_t-3(3)_t$ and
$t(4)_{t-1}=(4)_t-(3)_t$. Substituting (\ref{34Qans}) into
(\ref{34cdfexp}), we obtain
\begin{align*}
P\left(\lambda_{min}(\mathbf{W})>x\right){=}\exp(-\eta)& \text{etr}\left(-x\boldsymbol{\Sigma}^{-1}\right)
\left(\hspace{-1mm}\rho_1(x)\hspace{-1.5mm}\sum_{k=0}^\infty\hspace{-1mm}
\frac{\left(x\mu\right)^k}{k!(4)_k}\hspace{-1mm}\sum_{t=0}^k \hspace{-1mm}\binom{k}{t}(4)_t\hspace{-1mm}\left(\frac{\eta}{x\mu}\right)^t\right.\nonumber\\
& \qquad \left.+\;\rho_2(x)\sum_{k=0}^\infty
\frac{\left(x\mu\right)^k}{k!(4)_k}\sum_{t=0}^k \binom{k}{t}(3)_t\left(\frac{\eta}{x\mu}\right)^t\right).
\end{align*}
Finally, we re-sum the infinite series as power series in $x$ and
use (\ref{cdf}) to arrive at the result in (\ref{34}).
\end{proof}
\begin{remark}
An alternative form of the c.d.f.\ above can be written as
\begin{align*}
F_{{\text{min}}}(x)=
1-\exp\left(-\eta\right)\mathrm{etr}\left(-x\boldsymbol{\Sigma}^{-1}\right)&\left(
\rho_1(x)\Phi_3(4,4,\eta,x \mu)\right.\\
& \qquad \left.+ \rho_2(x)\Phi_3(3,4,\eta,x \mu)\right).
\end{align*}
\end{remark}

We now present some simulation results to verify the validity of our
new minimum eigenvalue distributions. We construct the covariance
$\boldsymbol{\Sigma}$ matrix with $(j,k)$th element
\begin{equation}
\label{covmatrix}
\boldsymbol{\Sigma}_{j,k}=\exp\left(-\frac{\pi^3}{32}(j-k)^2\right),\;\; 1\leq j,k\leq m
\end{equation}
and the mean matrix $\boldsymbol{\Upsilon}$ as
\begin{equation}
\boldsymbol{\Upsilon}=\mathbf{a}^H\mathbf{b}
\end{equation}
where
\begin{align*}
\mathbf{a}&=\left[ 1\;\exp\left(2i\pi\cos \theta\right)\; \exp\left(4i\pi\cos \theta\right)\ldots\; \exp\left(2(n-1)i\pi\cos \theta\right)\right]\\
\mathbf{b}&=\left[ 1\;\exp\left(2i\pi\cos \theta\right)\;
\exp\left(4i\pi\cos \theta\right)\ldots\; \exp\left(2(m-1)i\pi\cos
\theta\right)\right]
\end{align*}
 with $\theta=\pi/4$ and $i=\sqrt{-1}$.
Note that these particular constructions for the covariance and mean
matrices are employed since they are reasonable for modeling
practical correlated Rician MIMO channels \cite{MatthewIT,Bol}.

Fig. \ref{fig1} compares our analytical results with simulated data.
The analytical curves for the cases $m=n$ were calculated based on
Theorem \ref{th:MainResult}, while for the cases $m=2$ and $m=3$,
they were calculated based on Theorems \ref{th:nby2wishart} and
\ref{th:4by3wishart} respectively.  The accuracy of our results is
clearly evident from the figure.  Note that in evaluating these
analytical curves, the infinite summations in (\ref{cdfans}),
(\ref{cdfn2}), and (\ref{34}) were truncated to a maximum of 20
terms; thereby demonstrating a fast convergence rate for each
series.

\begin{figure}
 \centering
 \vspace*{-1.0cm}
     \subfigure[$n=m$]{
            \includegraphics[width=.8\textwidth]{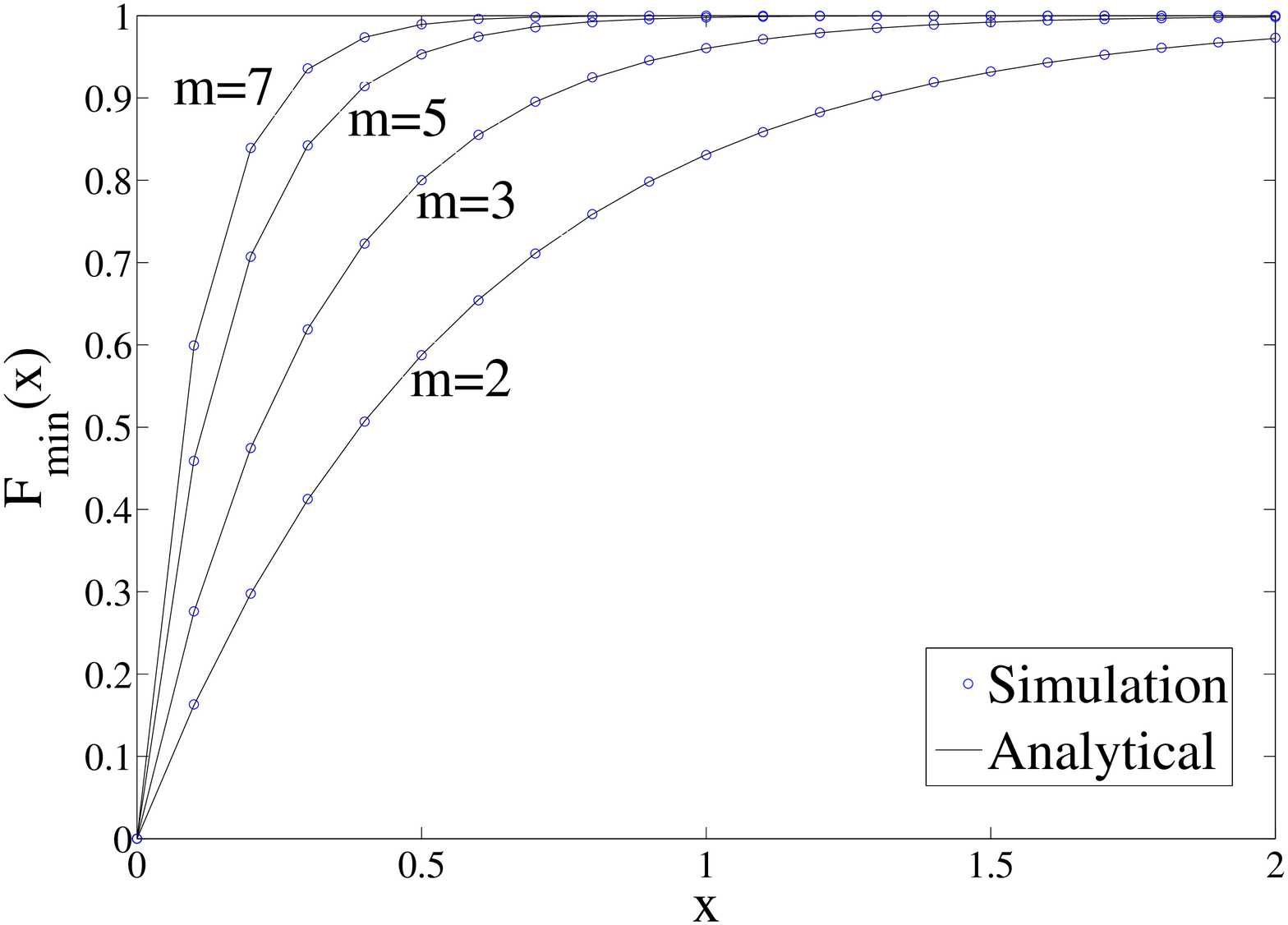}}
     \subfigure[$m=2,3$]{
          \includegraphics[width=.8\textwidth]{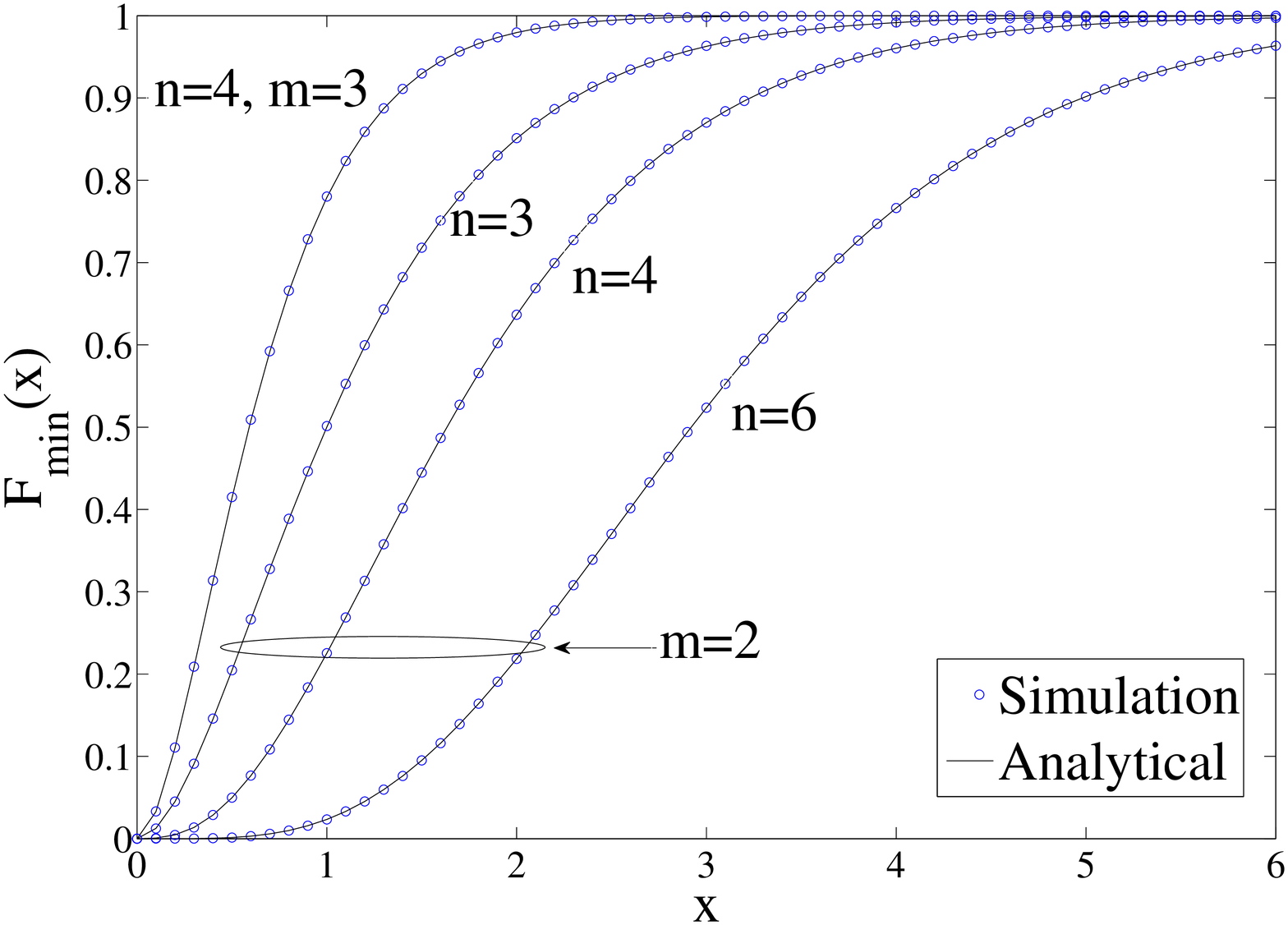}}\\
 \caption{Comparison of the analytical minimum eigenvalue c.d.f.s with simulated data points for correlated non-central Wishart matrices of various dimensions.}
 \vspace*{0.6cm}
 \label{fig1}
\end{figure}

\subsection{Gamma-Wishart Matrices}

We now turn to the analysis of the minimum eigenvalue distribution
of gamma-Wishart random matrices.  In this case, we deal with the matrix
$\mathbf{V}$ with joint density given in (\ref{Gram}). Thus, with
(\ref{eq:minEVRelation}), we have
\begin{equation*}
\label{cdfmin}
P\left( \lambda_{{min}}(\mathbf{V})>x\right)=
\mathcal{K}_{m,n}
\int_{\mathbf{V}-x\mathbf{I}_m\in\mathcal{H}_m^+}
\left|\mathbf{V}\right|^{n-m}\text{etr}\left(-\boldsymbol{\Sigma}^{-1}\mathbf{V}\right)
{}_1\widetilde{F}_1\left(\alpha;n;\mathbf{S}\mathbf{V}\right)d\mathbf{V}
\end{equation*}
where $\mathbf{S}=\boldsymbol{\Sigma}^{-1}\left(\boldsymbol{\Sigma}^{-1}+\boldsymbol{\Omega}\right)^{-1}
\boldsymbol{\Sigma}^{-1}$ and $\mathcal{K}_{m,n}=\frac{|\boldsymbol{\Omega}|^{\alpha}}
{\tilde \Gamma_m(n)|\boldsymbol{\Sigma}|^{n}\left|\boldsymbol{\Sigma}^{-1}+\boldsymbol{\Omega}\right|^{\alpha}}$ .
Applying the change of variables
$\mathbf{V}=x\left(\mathbf{I}_m+\mathbf{Y}\right)$ and using the
Kummer relation \cite{James1964}
\begin{equation*}
{}_1\tilde{F}_1\left(\alpha;n;x\mathbf{S}\left(\mathbf{I}_m+\mathbf{Y}\right)\right)=
\text{etr}\left(x\mathbf{S}\left(\mathbf{I}_m+\mathbf{Y}\right)\right)
{}_1\widetilde{F}_1\left(n-\alpha;n;-x\mathbf{S}\left(\mathbf{I}_m+\mathbf{Y}\right)\right)
\end{equation*}
yields
\begin{align}
\label{cdfseed}
P\left(\lambda_{{min}}(\mathbf{V})>x\right)=\mathcal{K}_{m,n}
x^{mn}\text{etr}\left(-x\mathbf{Q}\right)
&\int_{\mathbf{Y}\in\mathcal{H}_m^+}
 \left|\mathbf{I}_m+\mathbf{Y}\right|^{n-m}
  \text{etr}\left(-x\mathbf{Q}\mathbf{Y}\right)\nonumber\\
& \hspace*{-1cm} \times {}_1\widetilde{F}_1\left(n-\alpha;n;-x\mathbf{S}\left(\mathbf{I}_m+\mathbf{Y}\right)\right)d\mathbf{Y}
\end{align}
where $\mathbf{Q}=\boldsymbol{\Sigma}^{-1}-\mathbf{S}$.

This integral seems intractable for arbitrary values of $m$, $n$,
and $\alpha$. However, as we now show, it can be solved in closed
form solutions for some important configurations, thus yielding new exact expressions for the
minimum eigenvalue distributions.

The theorem below gives the exact minimum eigenvalue distribution for $2 \times 2$ gamma-Wishart matrices with arbitrary degrees of freedom (i.e., arbitrary $n$).

\begin{theorem}\label{th:wishgamnby2}
Let $\mathbf{V}\sim  \Gamma {\cal
W}_2 (n, \alpha, \boldsymbol{\Sigma}, \boldsymbol{\Omega})$, with
$\alpha \in \mathbb{Z}^{+}$ such that $\alpha>n\geq 2$. Then the c.d.f.\
of $\lambda_{min}(\mathbf{V})$ is given by
\begin{align}
\label{wishgamnby2}
F_{{{min}}}(x)=1-\mathcal{K}_{2,n}x^{2n}\mathrm{etr}\left(-x\mathbf{Q}\right)
 \sum_{k=0}^{2(\alpha-n)}
 \sum_{k_1=\left\lceil\frac{k}{2}\right\rceil}^{\min\left(k,\alpha-n\right)}
d_1^{k_1}\sum_{l=0}^{\left\lceil\frac{2k_1-k-1}{2}\right\rceil}d_2^{\kappa,l}\mathcal{I}_{k_1,l}(x)x^k
\end{align}
where
\begin{align*}d_1^{k_1}&=\frac{(\alpha-n)!(\alpha-n+1)!\left(2k_1-k+1\right)}{(\alpha-n-k_1)!
(\alpha-n+1+k_1-k)!\left(k_1+1\right)!\left(k-k_1\right)!(n)_{k_1}(n-1)_{k-k_1}}\\
d_2^{\kappa,l}&=(-1)^l4^l e^{\kappa}_l|\mathbf{S}|^{k-k_1+l} .
\end{align*}
Also,
\begin{align*}
\mathcal{I}_{k_1,l}(x)=\sum_{p=0}^{\varepsilon_{k_1,l}}  \sum_{j=0}^{\nu_{k_1,l}}
 p! \binom{\varepsilon_{k_1,l}}{p} \binom{\nu_{k_1,l}}{j}
\frac{ \mathrm{tr}^{\varepsilon_{k_1,l}-p}(\mathbf{S}) } { |\mathbf{Q}|^{j+2} x^{2(j+2)+p} }  \sum_{t=0}^j
\frac{j!}{(j-t)!}
|\mathbf{Q}|^t\mathcal{J}_{t,p,j}x^{t}\;,
\end{align*}
with
\begin{align*}
&\mathcal{J}_{t,p,j}=\sum_{t_1=\left\lceil\frac{t}{2}\right\rceil}^{t}
\tilde \Gamma_2(\omega_{j,t})\frac{\left(\omega_{j,t}\right)_{t_1}\left(\omega_{j,t}\right)_{t-t_1}\left(2t_1+1-t\right)}
{\left(t_1+1\right)!\left(t-t_1\right)!}
\sum_{i=0}^{\left\lceil\frac{2t_1-t-1}{2}\right\rceil}
\mathcal{L}_{\tau,p,i,j},
\end{align*}
where
\begin{align*}
\mathcal{L}_{\tau,p,i,j}{=}\sum_{q=0}^{\min(p,\varepsilon_{t_1,i})}
(-1)^{q+i}4^ie_i^\tau\binom{\varepsilon_{t_1,i}}{q}
\mathrm{tr}^{\varepsilon_{t_1,i}-q}(\mathbf{Q})& \mathrm{tr}^q(\mathbf{S})
|\mathbf{Q}|^{-\varepsilon_{t_1}-\frac{p-q}{2}}|\mathbf{S}|^{\frac{p-q}{2}}\\
& \times \mathcal{C}^{\varepsilon_{t_1}+\omega_{j,t}}_{p-q}\left(\frac{\mathrm{tr}\left(\mathbf{Q}^{-1}\mathbf{S}\right)}
{2\sqrt{\left|\mathbf{Q}^{-1}\mathbf{S}\right|}}\right) \; .
\end{align*}
$\kappa=(k_1,k-k_1)$ is a partition of $k$ such that $ \left\lceil\frac{k}{2}\right\rceil\leq k_1\leq \min(k,(\alpha-n))$, $\tau=(t_1,t-t_1)$ is a partition of $t$ such that $ \left\lceil\frac{t}{2}\right\rceil\leq t_1\leq t$, $\omega_{j,t}=j-t+2$ and $\nu_{k_1,l}=n+l+k-k_1-2$.
\end{theorem}
\begin{proof}
Particularizing (\ref{cdfseed}) to $m = 2$, $\alpha>n\geq 2$ and
$\alpha \in \mathbb{Z}^+$, and applying the zonal polynomial
expansion (\ref{hyptrk}) yields
\begin{align}
\label{hypozonexp}
P\left(\lambda_{{min}}(\mathbf{V})>x\right)=\mathcal{K}_{2,n}x^{2n}\text{etr}\left(-x\mathbf{Q}\right)
 \sum_{k=0}^{2(\alpha-n)}&\widetilde\sum_{\kappa}
\frac{[-(\alpha-n)]_{\kappa}}{[n]_\kappa k!}(-x)^k\nonumber\\
& \hspace*{-4cm} \times \int_{\mathbf{Y}\in\mathcal{H}_2^+}
\left|\mathbf{I}_2+\mathbf{Y}\right|^{n-2}\text{etr}\left(-x\mathbf{Q}\mathbf{Y}\right) C_{\kappa}\left(\mathbf{S}\left(\mathbf{I}_2+\mathbf{Y}\right)\right)d\mathbf{Y}
\end{align}
where $\kappa=(k_1,k_2)$ is a partition of $k$ into not more than two parts such that $k_1+k_2=k$ and $k_1\geq k_2 \geq 0,\;\forall k_1\in\{0,1,\ldots,\alpha-n\}$. Note that the series over $k$ is finite (truncated at $k=2(\alpha-n)$) due to the negative sign of the generalized complex hypergeometric coefficient. Careful inspection reveals that $\kappa$ can be written as $\kappa=\left(k_1,k-k_1\right)$, where $\left\lceil\frac{k}{2}\right\rceil\leq k_1\leq \min\left(k,(\alpha-n)\right)$.
This fact, along with the alternative representation of complex zonal polynomial given in \cite{Takemura,Mathai}, and Lemma \ref{lem:factorize},
\begin{align}
\label{zonal2exp}
C_{\kappa}\left(\mathbf{S}\left(\mathbf{I}_2+\mathbf{Y}\right)\right)=\;& \frac{k!\left(2k_1-k+1\right)}{\left(k_1+1\right)!\left(k-k_1\right)!}
\left|\mathbf{S}\left(\mathbf{I}_2+\mathbf{Y}\right)\right|^{k-k_1}\nonumber\\
& \hspace*{-1cm} \times \displaystyle \sum_{l=0}^{\left\lceil\frac{2k_1-k-1}{2}\right\rceil}
(-1)^l4^l e_l^\kappa
\left|\mathbf{S}\left(\mathbf{I}_2+\mathbf{Y}\right)\right|^{l}
\text{tr}^{\varepsilon_{k_1,l}}\left(\mathbf{S}\left(\mathbf{I}_2+\mathbf{Y}\right)\right)
\end{align}
gives (after some manipulations)
\begin{align}
\label{cdf2mexpress}
P\left(\lambda_{min}(\mathbf{V})>x\right)=\mathcal{K}_{2,n}x^{2n}\text{etr}\left(-x\mathbf{Q}\right)
 & \sum_{k=0}^{2(\alpha-n)}
\sum_{k_1=\left\lceil\frac{k}{2}\right\rceil}^{\min\left(k,(\alpha-n)\right)}
d_1^{k_1}\nonumber\\
& \quad \times \sum_{l=0}^{\left\lceil\frac{2k_1-k-1}{2}\right\rceil}d_2^{\kappa,l}\mathcal{I}_{k_1,l}(x)x^k
\end{align}
where
\begin{equation}
\mathcal{I}_{k_1,l}(x)=\int_{\mathbf{Y}\in\mathcal{H}_2^+}
\text{etr}\left(-x\mathbf{Q}\mathbf{Y}\right)\left|\mathbf{I}_2+\mathbf{Y}\right|^{\nu_{k_1,l}}
\text{tr}^{\varepsilon_{k_1,l}}\left(\mathbf{S}\left(\mathbf{I}_2+\mathbf{Y}\right)\right) d\mathbf{Y}.
\end{equation}
Using $\left|\mathbf{I}_2+\mathbf{Y}\right|=1+\text{tr}(\mathbf{Y})+|\mathbf{Y}|$
and the binomial theorem yields
\begin{align}
\label{Idef}
\mathcal{I}_{k_1,l}(x)= \sum_{p=0}^{\varepsilon_{k_1,l}}  \sum_{j=0}^{\nu_{k_1,l}}
 p!\binom{\varepsilon_{k_1,l}}{p} \binom{\nu_{k_1,l}}{j}
\frac{ \text{tr}^{\varepsilon_{k_1,l}-p}(\mathbf{S}) }{ |\mathbf{Q}|^{j+2} x^{2(j+2)+p} } \sum_{t=0}^j
\binom{j}{t}|\mathbf{Q}|^t\mathcal{J}_{t,p,j}x^{t}
\end{align}
where
\begin{equation}
\mathcal{J}_{t,p,j}=\frac{|x\mathbf{Q}|^{j-t+2}x^{t+p}}{p!}
\int_{\mathbf{Y}\in\mathcal{H}_2^+}
\text{etr}\left(-x\mathbf{Q}\mathbf{Y}\right)\text{tr}^{p}(\mathbf{SY})\text{tr}^{t}(\mathbf{Y})|\mathbf{Y}|^{j-t}d\mathbf{Y}.
\end{equation}
Finally, solving the remaining integral using Lemma \ref{lem:tracegamma} and recalling (\ref{cdf}) concludes the proof.
\end{proof}
Note that the minimum eigenvalue c.d.f.\ result given in (\ref{wishgamnby2}) can be easily computed numerically, since it contains only finite summations. Moreover, for specific values of $n$ and $\alpha$, it leads to simplified solutions, as shown in the following corollary.
\begin{corollary}

Let $\mathbf{V}\sim  \Gamma {\cal W}_2 (2, 3, \boldsymbol{\Sigma}, \boldsymbol{\Omega})$. Then the c.d.f. of $\lambda_{min}(\mathbf{V})$ is given by
\begin{align}
F_{min}(x)=1-\frac{|\boldsymbol{\Omega}|}{|\boldsymbol{\Sigma}^{-1}+\boldsymbol{\Omega}|}
&\mathrm{etr}\left(-x\mathbf{Q}\right)
 \left(\left|\mathbf{I}_2+\boldsymbol{\Omega}^{-1}\boldsymbol{\Sigma}^{-1}\right|+ \right.\nonumber\\
& \qquad \left.+\left(\frac{\mathrm{tr}(\mathbf{S})}{2}+\mathrm{tr}(\mathbf{Q}^{-1})|\mathbf{S}|\right)x+\frac{|\mathbf{S}|}{2}x^2\right).
\end{align}
\end{corollary}

The theorem below gives the exact minimum eigenvalue distribution for $3 \times 3$ gamma-Wishart matrices with $3$ degrees of freedom.
\begin{theorem}\label{th:4by3wishgama}
Let $\mathbf{V}\sim  \Gamma {\cal W}_3 (3, 4, \boldsymbol{\Sigma}, \boldsymbol{\Omega})$. Then the c.d.f. of $\lambda_{min}(\mathbf{V})$ is given by
\begin{align}
\label{cdf334}
F_{min}(x){=}1{-}\frac{|\boldsymbol{\Omega}|\mathrm{etr}\left(-x\mathbf{Q}\right)}{|\boldsymbol{\Sigma}^{-1}+\boldsymbol{\Omega}|}
&\left(\left|\mathbf{I}_3+\boldsymbol{\Omega}^{-1}\boldsymbol{\Sigma}^{-1}\right|+\mathrm{tr}(\mathbf{F})\frac{x}{6}+
\mathrm{tr}(\mathbf{G})|\mathbf{S}|\frac{x^2}{6} + |\mathbf{S}|\frac{x^3}{6}\right)
\end{align}
where
\begin{equation}
\mathbf{F}=2\mathbf{S}-3\mathbf{Q}^{-1}\mathbf{S}-3|\mathbf{S}|\mathbf{Q}^{-1}+6|\mathbf{S}||\mathbf{Q}|^{-1}\mathbf{Q}+3\left|\mathbf{I}_3+\mathbf{S}
\right|\mathbf{Q}^{-1}\left(\mathbf{I}_3+\mathbf{S}\right)^{-1}\mathbf{S}
\end{equation}
and
\begin{equation}
\mathbf{G}=\mathbf{S}^{-1}+3\mathbf{Q}^{-1} \; .
\end{equation}
\end{theorem}
\begin{proof}
In this case (\ref{cdfseed}) becomes
\begin{align}
P\left(\lambda_{{min}}(\mathbf{V})>x\right)=\mathcal{K}_{3,3}
x^{9} \text{etr}\left(-x\mathbf{Q}\right)
&\int_{\mathbf{Y}\in\mathcal{H}_3^+}
 \text{etr}\left(-x\mathbf{QY}\right)\nonumber\\
& \times {}_1\widetilde F_1\left(-1;3;-x\mathbf{S}\left(\mathbf{I}_3+\mathbf{S}\right)\right) d\mathbf{Y}
\end{align}
which upon applying the zonal polynomial expansion for the hypergeometric function (\ref{hyptrk}) yields
\begin{align}
\label{intmed}
P\left(\lambda_{{min}}(\mathbf{V})>x\right)=\mathcal{K}_{3,3}
x^{9} \text{etr}\left(-x\mathbf{Q}\right)
\sum_{k=0}^3
\frac{(-x)^k}{k!}& \widetilde\sum_{\kappa}\frac{[-1]_\kappa}{[3]_\kappa}
\int_{\mathbf{Y}\in\mathcal{H}_3^+}
\text{etr}\left(-x\mathbf{QY}\right)\nonumber\\
&  \times C_{\kappa}\left(\mathbf{S}\left(\mathbf{I}_3+\mathbf{Y}\right)\right)
d\mathbf{Y}
\end{align}
where $\kappa=\left(k_1,k_2,k_3\right)$ is a partition of $k$.
It is not difficult to see that the admissible partitions corresponding to the integers $0$, $1$, $2$, and $3$ are $(0,0,0)$, $(1,0,0)$, $(1,1,0)$, and $(1,1,1)$ respectively. Thus, we can write (\ref{intmed}) as
\begin{align}
\label{gamma3}
P\left(\lambda_{\text{min}}(\mathbf{V})>x\right)&=\mathcal{K}_{3,3}
x^{9} \text{etr}\left(-x\mathbf{Q}\right)
\left(
\int_{\mathbf{Y}\in\mathcal{H}_3^+}
\text{etr}\left(-x\mathbf{QY}\right) d\mathbf{Y}\right.\nonumber\\
& \qquad \left. +
\frac{x}{3}
\int_{\mathbf{Y}\in\mathcal{H}_3^+}
\text{etr}\left(-x\mathbf{QY}\right)
C_{1,0,0}\left(\mathbf{S}\left(\mathbf{I}_3+\mathbf{Y}\right)\right)
d\mathbf{Y}\right.\nonumber\\
& \qquad \left.
 +
\frac{x^2}{6}\int_{\mathbf{Y}\in\mathcal{H}_3^+}
\text{etr}\left(-x\mathbf{QY}\right)
C_{1,1,0}\left(\mathbf{S}\left(\mathbf{I}_3+\mathbf{Y}\right)\right)
d\mathbf{Y}\right.\nonumber\\
&\quad \left.
+ \frac{x^3}{6}
\int_{\mathbf{Y}\in\mathcal{H}_3^+}
\text{etr}\left(-x\mathbf{QY}\right)
C_{1,1,1}\left(\mathbf{S}\left(\mathbf{I}_3+\mathbf{Y}\right)\right)
d\mathbf{Y}
\right).
\end{align}
Moreover, we have
\begin{align}
\label{zonal3iden}
C_{1,0,0}\left(\mathbf{S}\left(\mathbf{I}_3+\mathbf{Y}\right)\right)&=\text{tr}\left(\mathbf{S}\left(\mathbf{I}_3+\mathbf{Y}\right)\right)
 \nonumber\\
C_{1,1,1}\left(\mathbf{S}\left(\mathbf{I}_3+\mathbf{Y}\right)\right)&=|\mathbf{S}||\mathbf{I}_3+\mathbf{Y}|.
\end{align}
Utilizing (\ref{det3ident}) we can express
\begin{equation}
\label{zonal32}
C_{1,1,0}\left(\mathbf{S}\left(\mathbf{I}_3+\mathbf{Y}\right)\right)=\left|\mathbf{I}_3+\left(\mathbf{S}\left(\mathbf{I}_3+\mathbf{Y}\right)\right)\right|-1-
\text{tr}\left(\mathbf{S}\left(\mathbf{I}_3+\mathbf{Y}\right)\right)-|\mathbf{S}\left(\mathbf{I}_3+\mathbf{Y}\right)|.
\end{equation}
Now, using (\ref{zonal3iden}) and (\ref{zonal32}) in (\ref{gamma3}) yields
\begin{align}
\label{h1h2def}
P\left(\lambda_{min}(\mathbf{V})>x\right)=\mathcal{K}_{3,3} &
x^{9}  \text{etr}\left(-x\mathbf{Q}\right) \left(
\left(1-\frac{x^2}{6}-\frac{\text{tr}(\mathbf{S})x^2}{6}+\frac{\text{tr}(\mathbf{S})x}{3}\right)\right.\nonumber\\
& \qquad \times \int_{\mathbf{Y}\in\mathcal{H}_3^+}
\text{etr}\left(-x\mathbf{QY}\right) d\mathbf{Y}\nonumber\\
& +
\left(\frac{x}{3}-\frac{x^2}{6}\right)
\int_{\mathbf{Y}\in\mathcal{H}_3^+}
\text{etr}\left(-x\mathbf{QY}\right) C_{1,0,0}\left(\mathbf{SY}\right)d\mathbf{Y}
\nonumber\\
& \left.
+\frac{|\mathbf{S}|}{6}\left(x^3-x^2\right)\mathcal{G}_1(x)+\left|\mathbf{I}_3+\mathbf{S}\right|\frac{x^2}{6}\mathcal{G}_2(x)
\right)
\end{align}
where
\begin{equation}
\label{h1}
\mathcal{G}_1(x)=\int_{\mathbf{Y}\in\mathcal{H}_3^+}
\text{etr}\left(-x\mathbf{QY}\right) \left|\mathbf{I}_3+\mathbf{Y}\right|d\mathbf{Y}
\end{equation}
and
\begin{equation}
\label{h2}
\mathcal{G}_2(x)=\int_{\mathbf{Y}\in\mathcal{H}_3^+}
\text{etr}\left(-x\mathbf{QY}\right) \left|\mathbf{I}_3+\left(\mathbf{I}_3+\mathbf{S}\right)^{-1}\mathbf{S}\mathbf{Y}\right|d\mathbf{Y}.
\end{equation}
The first and second integrals in (\ref{h1h2def}) can be evaluated using \cite[Eq. 6.1.20]{Mathai2}, thus we concentrate on the evaluation of $\mathcal{G}_1(x)$ and $\mathcal{G}_2(x)$. We provide a detailed solution for the integral $\mathcal{G}_2(x)$ only, since both (\ref{h1}) and (\ref{h2}) share a common structure.

Using the relation $\left|\mathbf{I}_3+\left(\mathbf{I}_3+\mathbf{S}\right)^{-1}\mathbf{S}\mathbf{Y}\right|={}_1\widetilde F_0\left(-1;-\left(\mathbf{I}_3+\mathbf{S}\right)^{-1}\mathbf{S}\mathbf{Y}\right)$
in (\ref{h2}) yields
\begin{equation}
\mathcal{G}_2(x)=\int_{\mathbf{Y}\in\mathcal{H}_3^+}
\text{etr}\left(-x\mathbf{QY}\right)
{}_1\widetilde F_0\left(-1;-\left(\mathbf{I}_3+\mathbf{S}\right)^{-1}\mathbf{S}\mathbf{Y}\right) d\mathbf{Y}.
\end{equation}
This integral can be solved using \cite[Eq. 3.20]{Rathna} as
\begin{align*}
\mathcal{G}_2(x)& =\tilde \Gamma_3(3)|\mathbf{Q}|^{-3}x^{-9}\;
{}_2\widetilde F_{0}\left(-1,3;-x^{-1}\mathbf{Q}^{-1}\left(\mathbf{I}_3+\mathbf{S}\right)^{-1}\mathbf{S}\right)\nonumber\\
& = \tilde \Gamma_3(3)|\mathbf{Q}|^{-3}x^{-9}
\sum_{k=0}^3\frac{(-1)^k}{x^k k!}
\widetilde \sum_{\kappa}
[-1]_\kappa [3]_\kappa
C_\kappa\left(\mathbf{Q}^{-1}\left(\mathbf{I}_3+\mathbf{S}\right)^{-1}\mathbf{S}\right).
\end{align*}
Since the valid partitions corresponding to the summation index $k=0,1,2$ and $3$ are respectively $(0,0,0),(1,0,0),(1,1,0)$ and $(1,1,1)$, we can use equations analogous to (\ref{zonal3iden}) to obtain
\begin{align}
\label{h2ans}
\mathcal{G}_2(x)  =\tilde \Gamma_3(3)|\mathbf{Q}|^{-3}& x^{-9}
\left(1+3x^{-1}\text{tr}\left(\mathbf{Q}^{-1}\left(\mathbf{I}_3+\mathbf{S}\right)^{-1}\mathbf{S}\right)\right.\nonumber\\
& \quad +\;6x^{-2}C_{1,1,0}\left(\mathbf{Q}^{-1}\left(\mathbf{I}_3+\mathbf{S}\right)^{-1}\mathbf{S}\right)\nonumber\\
& \quad \left.+\;\;6 x^{-3}|\mathbf{Q}|^{-1}\left|\mathbf{I}_3+\mathbf{S}\right|^{-1}|\mathbf{S}|\right).
\end{align}
Following similar arguments, we can obtain
\begin{align}
\label{h1ans}
\mathcal{G}_1(x) & =\tilde \Gamma_3(3)|\mathbf{Q}|^{-3}x^{-9}
\left(
1+3x^{-1}\text{tr}\left(\mathbf{Q}^{-1}\right)+6x^{-2}C_{1,1,0}\left(\mathbf{Q}^{-1}\right)+6x^{-3}|\mathbf{Q}|^{-1}
\right).
\end{align}
Finally, using (\ref{h2ans}), (\ref{h1ans}), and (\ref{110zoanlex}) in (\ref{h1h2def}), recalling (\ref{cdf}), and applying  some lengthy algebraic manipulations, we arrive at the result in (\ref{cdf334}).
\end{proof}

Fig. \ref{fig2} compares our analytical results with simulated data.  The analytical curves for the cases  $m=2$ and $m=3$ were computed based on Theorems \ref{th:wishgamnby2} and \ref{th:4by3wishgama} respectively. Here we have used the same $\boldsymbol{\Sigma}$ as defined in (\ref{covmatrix}), whereas $\boldsymbol{\Omega}$ is constructed with the following $j,k$th element:
\begin{equation}
\label{gampara}
\boldsymbol{\Omega}_{j,k}=\exp\left(-0.7(j-k)i\pi\right)\exp\left(-\frac{147\pi^3}{4000}(j-k)^2\right),\;\; 1\leq j,k\leq m
\end{equation}
with $i=\sqrt{-1}$. As expected, the analytical curves match closely with the simulated curves.

\begin{figure}
 \centering
 \vspace*{-1.0cm}
     \subfigure[$n=3, m=2$]{
            \includegraphics[width=.8\textwidth]{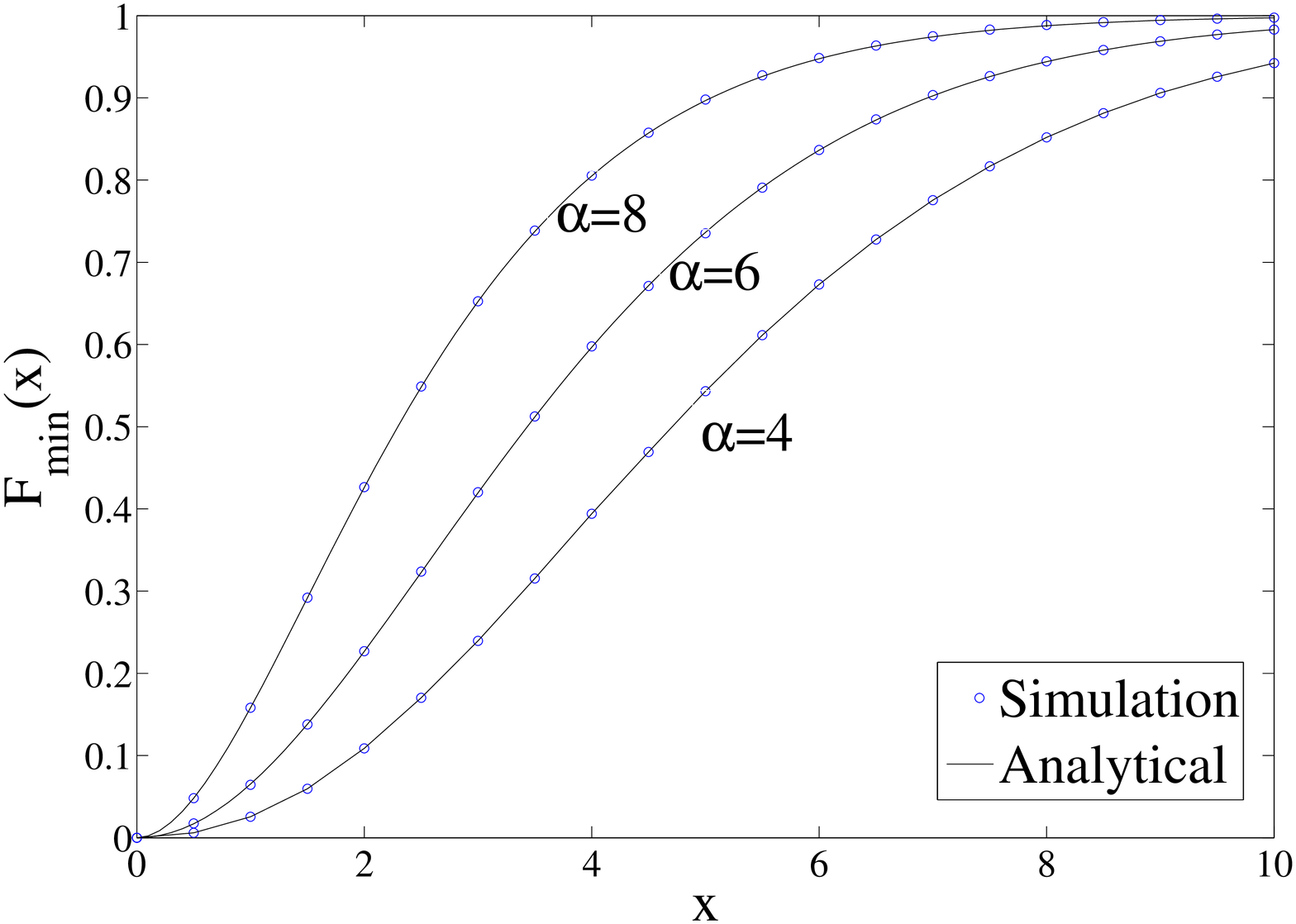}}
     \subfigure[$m=2,3$]{
          \includegraphics[width=.8\textwidth]{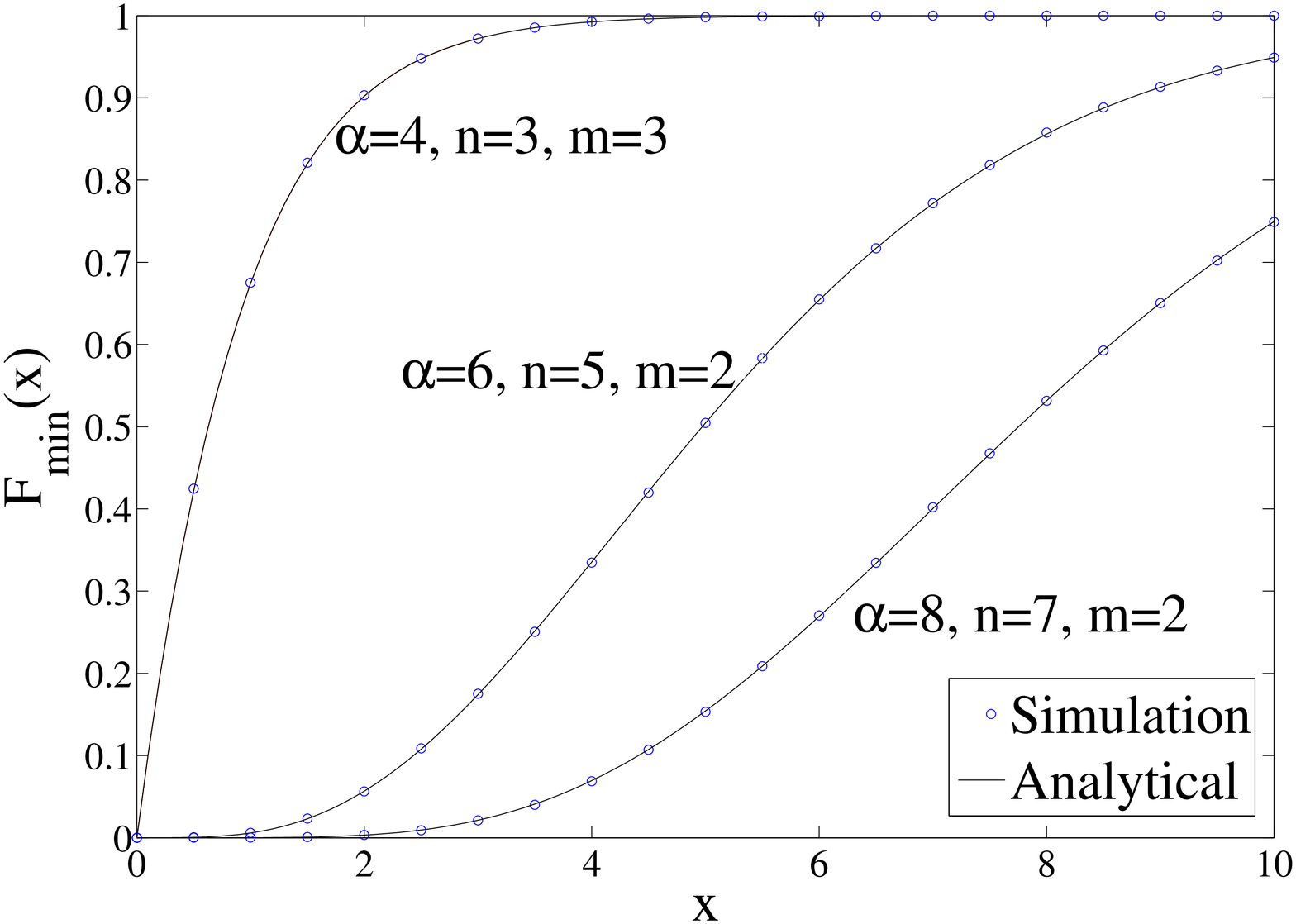}}\\
 \caption{Comparison of the analytical minimum eigenvalue c.d.f.s with simulated data points for correlated gamma-Wishart matrices with various dimensions and parameters.}
 \label{fig2}
\end{figure}

\section{New Maximum Eigenvalue Distributions}

In this section, we shift attention to the distribution of the
\emph{maximum} eigenvalue of correlated non-central Wishart and
gamma-Wishart random matrices.  As for the minimum eigenvalue
distribution considered previously, once again the most direct
approach of integrating the joint eigenvalue p.d.f. over a suitable multidimensional region seems intractable. To this end, we write the maximum eigenvalue $\lambda_{max} (\mathbf{Y})$ of $\mathbf{Y}\in\mathcal{H}_m^+$ as
\begin{equation}
\label{cdfmax}
F_{max}(x)=P\left(\lambda_{max}(\mathbf{Y})<
x\right)=P\left(\mathbf{Y} < x
\mathbf{I}_m \right)  \;
\end{equation}
which allows one to deal purely with the distribution of
$\mathbf{Y}$, rather than the distribution of its eigenvalues.

\subsection{Correlated Non-Central Wishart Case}

For the non-central Wishart scenario, we deal with the matrix
$\mathbf{W}$ with joint density given in (\ref{wishart}). Thus, with
(\ref{cdfmax}), we have
\begin{align}
\label{maximumeigen}
P\left(\lambda_{max}(\mathbf{W})<x\right)&=\int_{\mathbf{W}<x\mathbf{I}_m}
f_{\mathbf{W}}(\mathbf{W})d\mathbf{W}\nonumber\\
&= \frac{\exp(-\eta)}{\tilde \Gamma_m(n) |\boldsymbol{\Sigma}|^n}
\int_{x\mathbf{I}_m-\mathbf{W}\in\mathcal{H}_m^+}|\mathbf{W}|^{n-m}\text{etr}\left(-\boldsymbol{\Sigma}^{-1}\mathbf{W}\right)\nonumber\\
& \hspace{3cm} \qquad \times {}_0\widetilde F_1\left(n;\boldsymbol{\Theta\Sigma}^{-1}\mathbf{W}\right)d\mathbf{W}.
\end{align}
Applying the change of variable $\mathbf{W}=x\mathbf{Y}$ with
$d\mathbf{W}=x^{m^2}d\mathbf{Y}$ in (\ref{maximumeigen}) gives
\begin{align}
P\left(\lambda_{max}(\mathbf{W})<x\right)=
\frac{x^{mn}\exp(-\eta)}{\tilde \Gamma_m(n) |\boldsymbol{\Sigma}|^n}
\int_{\mathbf{0}}^{\mathbf{I}_m}
& |\mathbf{Y}|^{n-m} \text{etr}\left(-x\boldsymbol{\Sigma}^{-1}\mathbf{Y}\right)\nonumber\\
& \times {}_0\widetilde F_1\left(n;x\boldsymbol{\Theta\Sigma}^{-1}\mathbf{Y}\right)d\mathbf{Y}.
\end{align}
Expanding the hypergeometric function with its equivalent series
expansion followed by using the reasoning which led to
(\ref{zonaldef}) yields
\begin{align}
\label{maxwishart}
P\left(\lambda_{max}(\mathbf{W})<x\right)=
\frac{x^{mn}\exp(-\eta)}{\tilde \Gamma_m(n) |\boldsymbol{\Sigma}|^n}
\sum_{k=0}^\infty
\frac{\left(x\mu\right)^k}{(n)_k k!}
\int_{\mathbf{0}}^{\mathbf{I}_m}
& |\mathbf{Y}|^{n-m}\text{etr}\left(-x\boldsymbol{\Sigma}^{-1}\mathbf{Y}\right)\nonumber\\
& \times \text{tr}^k\left(\boldsymbol{\alpha}\boldsymbol{\alpha}^H\mathbf{Y}\right)d\mathbf{Y}
\end{align}
where we have applied
$\left(\boldsymbol{\alpha}^H\mathbf{Y}\boldsymbol{\alpha}\right)^k=\text{tr}^k\left(\boldsymbol{\alpha}\boldsymbol{\alpha}^H\mathbf{Y}\right)$.
This matrix integral seems intractable for arbitrary values $m$ and
$n$.  In fact, this integral seems even more difficult to tackle
than that which arises in the minimum eigenvalue formulation, i.e.,
Eq.\ (\ref{finalmatintegra}).  As the following theorem shows,
however, we can obtain a solution for the case of $2 \times 2$
non-central Wishart matrices with arbitrary degrees of freedom. This
is significant, because it presents the first tractable result for
the maximum eigenvalue c.d.f.\ of correlated complex non-central
Wishart matrices.

\begin{theorem}
Let $\mathbf{X}\sim \mathcal{CN}_{n,2}\left(\boldsymbol{\Upsilon},\mathbf{I}_n\otimes\boldsymbol{\Sigma}\right)$, where $\boldsymbol{\Upsilon}\in\mathbb{C}^{n\times 2}$ has rank one, and $\mathbf{W}=\mathbf{X}^H\mathbf{X}$. Then the c.d.f.\ of $\lambda_{max}(\mathbf{W})$ is given by
\begin{equation} \label{eq:nonCentMax}
F_{max}(x)=\frac{x^{2n}\exp(-\eta)}{n!(n+1)!}
\sum_{k=0}^\infty
\frac{\left(x\mu\right)^k}{(n)_k k!}
\phi_{-x\boldsymbol{\Sigma}^{-1},\boldsymbol{\alpha}\boldsymbol{\alpha}^H,n}^{(k)}(0)
\end{equation}
where
$\phi_{-x\boldsymbol{\Sigma}^{-1},\boldsymbol{\alpha}\boldsymbol{\alpha}^H,n}^{(k)}(0)$
is calculated recursively via (\ref{sub1})-(\ref{initialcond}).
\end{theorem}
\begin{proof}
Substituting $m=2$ into (\ref{maxwishart}), the proof follows upon
application of Lemma \ref{lem:1f1}.
\end{proof}
\begin{remark}
An alternative expression for (\ref{eq:nonCentMax}) can be obtained
by employing the moment generating function based power series
expansion approach given in \cite{Mathai3}. However, we have found
that by employing that approach the final expression is more
complicated, since it includes two infinite summations along with a
recursive summation term.
\end{remark}

\subsection{Correlated Gamma-Wishart Case}

We now turn consider the maximum eigenvalue distribution of
gamma-Wishart random matrices.  In this case, we deal with the
matrix $\mathbf{V}$ with joint density given in (\ref{Gram}). Thus,
with (\ref{cdfmax}), we have
\begin{align}
\label{grammax}
P\left(\lambda_{max}(\mathbf{V})<x\right)=
\mathcal{K}_{m,n} x^{mn}
\int_0^{\mathbf{I}_m}
|\mathbf{Y}|^{n-m}&\text{etr}\left(-x\mathbf{QY}\right)\nonumber\\
&\times {}_1\widetilde F_1\left(n-\alpha;n;-x\mathbf{SY}\right)d\mathbf{Y}.
\end{align}
In the following theorem, we present a new exact closed form
expression for the c.d.f.\ of the maximum eigenvalue of $\mathbf{V}$
for some particularizations of $m$, $n$ and $\alpha$.
\begin{theorem}\label{th:maxwishgam}
Let $\mathbf{V}\sim  \Gamma {\cal
W}_2 (n, \alpha, \boldsymbol{\Sigma}, \boldsymbol{\Omega})$ with $\alpha>n\geq 2$. Then the c.d.f.\
of $\lambda_{max}(\mathbf{V})$ is given by
\begin{align}
\label{anmax2}
F_{max}(x)&=
\mathcal{K}_{2,n}x^{2n}
\sum_{k=0}^{2(\alpha-n)}
\sum_{k_1=\left\lceil\frac{k}{2}\right\rceil}^{\min\left(k,(\alpha-n)\right)}
d_1^{k_1}
\sum_{l=0}^{\left\lceil\frac{2k_1-k-1}{2}\right\rceil} d_2^{\kappa,l}\mathcal{R}_{k_1,l}(x)x^k
\end{align}
where
\begin{equation}
\label{defR}
\mathcal{R}_{k_1,l}(x)=\frac{\tilde \Gamma_2(2)\tilde \Gamma_2\left(\nu_{k_1,l}+2\right)}{\tilde \Gamma \left(\nu_{k_1,l}+4\right)}
\phi_{-x\mathbf{Q},\mathbf{S},\nu_{k_1,l}+2}^{(\varepsilon_{k_1,l})}(0),
\end{equation}
$\varepsilon_{k_1,l}=2k_1-k-2l$, $\nu_{k_1,l}=n+l+k-k_1-2$,
$\kappa=\left(k_1,k-k_1\right)$ is a partition of $k$ such that
$k_1\in\left\{0,1,\ldots,(\alpha-n)\right\}$ and
$\left\lceil\frac{k}{2}\right\rceil\leq
k_1\leq\min\left(k,(\alpha-n)\right)$.  The term
$\phi_{-x\mathbf{Q},\mathbf{S},\nu_{k_1,l}+2}^{(\varepsilon_{k_1,l})}(0)$
is calculated recursively via (\ref{sub1})-(\ref{initialcond}).
\end{theorem}
\begin{proof}
Particularizing (\ref{grammax}) to  $m=2$, $\alpha>n\geq 2$ and
$\alpha \in \mathbb{Z}^+$ and applying the zonal polynomial
expansion (\ref{hyptrk}) gives
\begin{align*}
F_{max}(x)=
\mathcal{K}_{2,n}x^{2n}
\sum_{k=0}^{2(\alpha-n)}
\widetilde \sum_{\kappa}
\frac{[-(\alpha-n)]_\kappa}{[n]_\kappa}\frac{(-x)^k}{k!}
 \int_0^{\mathbf{I}_2}
& |\mathbf{Y}|^{n-2}\text{etr}\left(-x\mathbf{QY}\right)\nonumber\\
& \qquad \times C_\kappa(\mathbf{SY})d\mathbf{Y}
\end{align*}
Following the similar reasoning which led to (\ref{cdf2mexpress}),
with some algebraic manipulations we obtain (\ref{anmax2}), but with
\begin{equation*}
\mathcal{R}_{\kappa,l}(x)=
\int_0^{\mathbf{I}_2}
\text{etr}\left(-x\mathbf{QY}\right)|\mathbf{Y}|^{\nu_{k_1,l}}\text{tr}^{\varepsilon_{k_1,l}}(\mathbf{SY})d\mathbf{Y}.
\end{equation*}
This integrals is solved via Lemma \ref{lem:1f1} to yield
(\ref{defR}).
\end{proof}

Note that the c.d.f.\ result in Theorem \ref{th:maxwishgam} can be
evaluated numerically for any value of $n$. Moreover, for specific
values of $n$ it often gives simplified solutions.  Some examples
are shown in the following corollaries.

\begin{corollary}
Let $\mathbf{V}\sim  \Gamma {\cal
W}_2 (n, n+1, \boldsymbol{\Sigma}, \boldsymbol{\Omega})$. Then the
c.d.f.\ of $\lambda_{max}(\mathbf{V})$ is given by
\begin{align}
\label{maxgram}
F_{max}(x)&=
\frac{|\boldsymbol{\Omega}|^{n+1}x^{2n}}{n!(n+1)!|\boldsymbol{\Sigma}|^n|\boldsymbol{\Sigma}^{-1}+\boldsymbol{\Omega}|^{n+1}}
\left({}_1\widetilde F_1\left(n;n+2;-x\mathbf{Q}\right)+\frac{x}{n}\phi_{-x\mathbf{Q},\mathbf{S},n}^{(1)}(0)\right.\nonumber\\
&\hspace{3cm}\left.+
\frac{|\mathbf{S}|x^2}{(n+1)(n+2)}{}_1\widetilde F_1\left(n+1;n+3;-x\mathbf{Q}\right)\right).
\end{align}
\end{corollary}
\begin{corollary}
Let $\mathbf{V}\sim  \Gamma {\cal
W}_2 (n, n+2, \boldsymbol{\Sigma}, \boldsymbol{\Omega})$. Then the
c.d.f.\ of $\lambda_{max}(\mathbf{V})$ is given by
\begin{align}
F_{max}(x)&=
\frac{|\boldsymbol{\Omega}|^{n+2}x^{2n}}{n!(n+1)!|\boldsymbol{\Sigma}|^n|\boldsymbol{\Sigma}^{-1}+\boldsymbol{\Omega}|^{n+2}}
\left({}_1\widetilde F_1\left(n;n+2;-x\mathbf{Q}\right){+}\frac{2x}{n}\phi_{-x\mathbf{Q},\mathbf{S},n}^{(1)}(0)\right.\nonumber\\
&\qquad \qquad+\frac{x^2}{n(n+1)}\phi_{-x\mathbf{Q},\mathbf{S},n}^{(2)}(0)+\frac{2|\mathbf{S}|x^2}{(n+1)^2}{}_1\widetilde F_1\left(n+1;n+3;-x\mathbf{Q}\right)\nonumber\\
& \qquad\qquad+ \frac{2|\mathbf{S}|x^3}{(n+1)^2(n+2)}\phi_{-x\mathbf{Q},\mathbf{S},n+1}^{(1)}(0)\nonumber\\
& \qquad\quad\left.+
\frac{|\mathbf{S}|^2x^4}{(n+1)(n+2)^2(n+3)}{}_1\widetilde F_1\left(n+2;n+4;-x\mathbf{Q}\right)\right).
\end{align}
\end{corollary}

\begin{figure}[h]
\centering
\includegraphics[width=.8\textwidth]{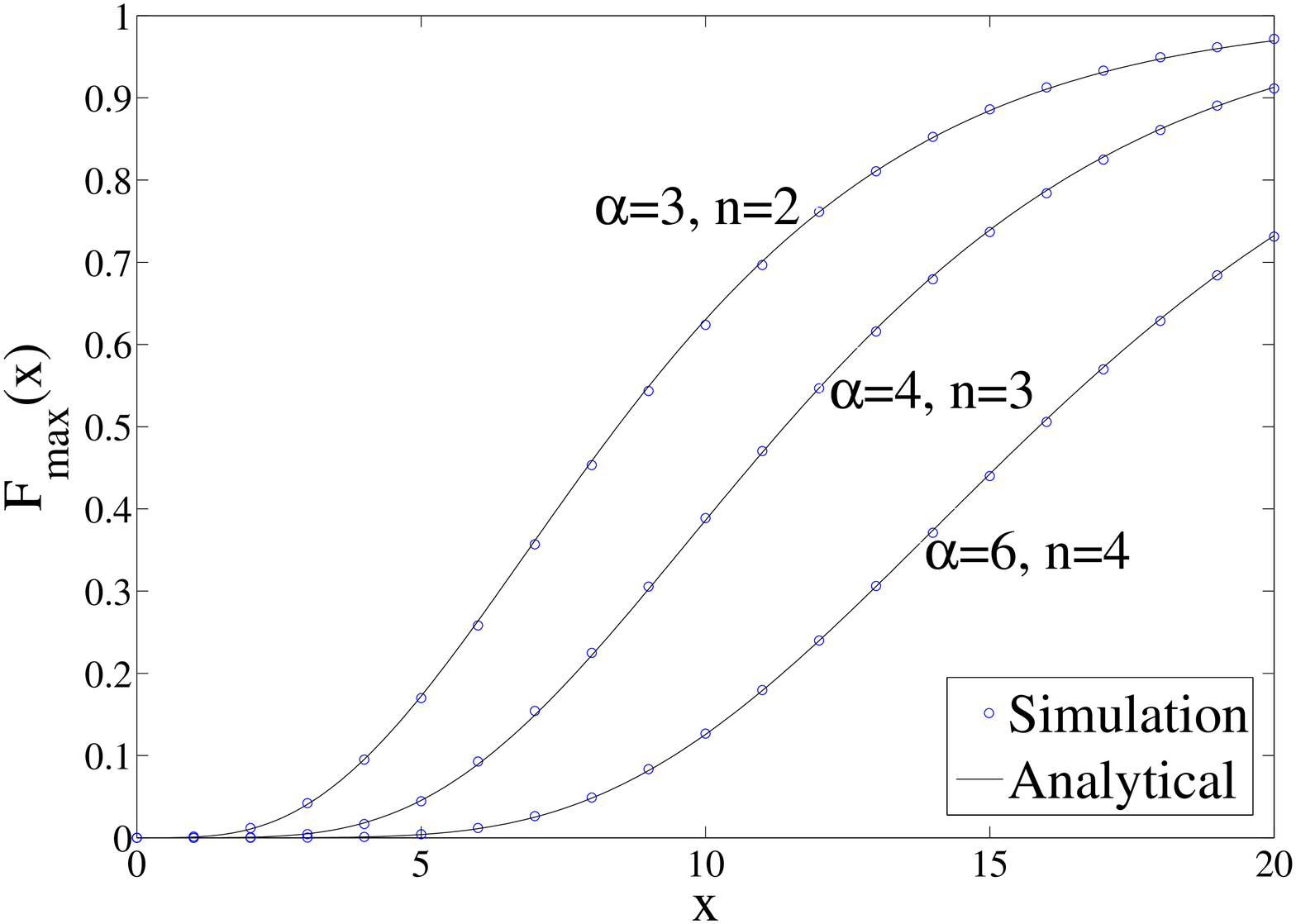}
\caption{Comparison of the analytical maximum eigenvalue c.d.f.s with simulated data points for correlated gamma-Wishart matrices with various dimensions and parameters.}
\label{fig3}
\end{figure}

Fig. \ref{fig3} compares the analytical c.d.f.\ results for the
maximum eigenvalue of gamma-Wishart matrices with simulated data.
the matrix $\boldsymbol{\Sigma}$ and $\boldsymbol{\Omega}$ are
constructed as in (\ref{covmatrix}) and (\ref{gampara})
respectively. The analytical curves were computed based on Theorem
\ref{th:maxwishgam}. The agreement between the analysis and
simulation is clearly evident.

\section{Conclusions}
We have derived new exact closed-form expressions for the c.d.f.\ of the extreme eigenvalues of correlated complex non-central Wishart and gamma-Wishart random matrices. We would like to conclude by emphasizing that these results provide the first tractable exact analytical results pertaining to the eigenvalue distributions of both complex non-central Wishart and gamma-Wishart random matrices with non-trivial correlation structures. Obtaining tractable solutions for extreme eigenvalue densities for generalized parameters (e.g., for arbitrary matrix dimensions) remains an important open problem.

\appendix
\section{Proof of Lemma \ref{lem:factorize}} \label{ap:A}
\begin{proof}
We start by factorizing $x_1^n-x_2^n$ and using
$x_1+x_2=\text{tr}(\mathbf{X})$ and $x_1x_2=|\mathbf{X}|$ to obtain
\begin{equation}
\label{factx}
\frac{x_1^n-x_2^n}{x_1-x_2}=\left\{
\begin{array}{cl}
\text{tr}(\mathbf{X})\displaystyle \prod_{j=1}^{\frac{n-2}{2}}\left(\text{tr}^2(\mathbf{X})-4|\mathbf{X}|\cos^2\left(\frac{\pi j}{n}\right)\right)& \text{even $n$}\\
\displaystyle \prod_{j=1}^{\frac{n-1}{2}}\left(\text{tr}^2(\mathbf{X})-4|\mathbf{X}|\cos^2\left(\frac{\pi j}{n}\right)\right)& \text{odd $n$}.
\end{array}\right.
\end{equation}
Next, recalling the generating function expansion
\begin{equation}
\label{permanent}
\prod_{j=1}^n\left(x-\psi_j y\right)=\sum_{j=0}^n(-1)^j
e_j x^{n-j}y^j
\end{equation}
where $e_i$ denotes the $i$th elementary symmetric function
\cite{Mac} of the parameters $\{\psi_1,\psi_2,\ldots,\psi_n\}$, and
using (\ref{permanent}) in (\ref{factx}) along with some algebra, we
obtain the result.
\end{proof}
\section{Proof of Lemma \ref{lem:1f1}} \label{ap:B}
Using \cite[Eq. 3.23]{Rathna}, we have
\begin{equation}
\label{1f1def}
{}_1\widetilde F_1\left(a;a+2;\mathbf{X}\right)=\mathcal{K}\int_\mathbf{0}^{\mathbf{I}_2}
|\mathbf{Z}|^{a-2}\text{etr}\left(\mathbf{XZ}\right)d\mathbf{Z}
\end{equation}
where $\mathcal{K}=\frac{\tilde \Gamma_2(a+2)}{\tilde
\Gamma_2(a)\tilde \Gamma_2(2)}$ and $\Re(a)>1$. Following the proof
of \cite[Lemma 7]{Khatri1966}, we substitute
$\mathbf{X}=\mathbf{A}+\mathbf{B}y$ into (\ref{1f1def}) to yield
\begin{equation}
{}_1\widetilde F_1\left(a;a+2;\mathbf{A}+\mathbf{B}y\right)=
\mathcal{K}
\int_\mathbf{0}^{\mathbf{I}_2}
|\mathbf{Z}|^{a-2}
\text{etr}\left(\left(\mathbf{A}+\mathbf{B}y\right)\mathbf{Z}\right)d\mathbf{Z}
\end{equation}
where $y\geq 0$. Expanding the term $\text{etr}(\mathbf{BZ}y)$ gives
\begin{equation}
\label{1f1seed}
{}_1\widetilde F_1\left(a;a+2;\mathbf{A}+\mathbf{B}y\right)=
\mathcal{K}
\sum_{p=0}^\infty
\frac{y^p}{p!}
\int_\mathbf{0}^{\mathbf{I}_2}
|\mathbf{Z}|^{a-2}
\text{etr}\left(\mathbf{A}\mathbf{Z}\right)\text{tr}^{p}\left(\mathbf{BZ}\right)d\mathbf{Z}.
\end{equation}
Now, we aim to establish a power series expansion for
${}_1\widetilde F_1\left(a;a+2;\mathbf{A}+\mathbf{B}y\right)$ around
$y=0$. To this end, denote
\begin{equation}
\phi_{\mathbf{A},\mathbf{B},a}(y)={}_1\widetilde
F_1\left(a;a+2;\mathbf{A}+\mathbf{B}y\right) \; .
\end{equation}
We then have
\begin{equation}
\label{1f1series}
{}_1\widetilde F_1\left(a;a+2;\mathbf{A}+\mathbf{B}y\right)=
\sum_{p=0}^\infty
\frac{y^p}{p!}\phi^{(p)}_{\mathbf{A},\mathbf{B},a}(0).
\end{equation}
Equating the coefficient of $y^p$ on both sides of (\ref{1f1seed})
and (\ref{1f1series}) gives (\ref{incomgammaint}).

Following \cite{Orlov,Gross}, we can express the confluent hypergeometric
function of a matrix argument in the determinant form
\begin{equation}
\label{phidef}
\phi_{\mathbf{A},\mathbf{B},a}(y)=\frac{\Delta_{\mathbf{A},\mathbf{B},a}(y)}{h_{\mathbf{A},\mathbf{B}}(y)}
\end{equation}
where $h_{\mathbf{A},\mathbf{B}}(y)=x_1(y)-x_2(y)$. Since we are interested in obtaining $\phi^{(p)}_{\mathbf{A},\mathbf{B},a}(0)$, we may rearrange (\ref{phidef}) such that
\begin{equation*}
h_{\mathbf{A},\mathbf{B}}(y)\phi_{\mathbf{A},\mathbf{B},a}(y)=\Delta_{\mathbf{A},\mathbf{B},a}(y)
\end{equation*}
and apply Leibniz's rule \cite{Gradshteyn} for the $k$th derivative of a product to obtain
\begin{equation}
\sum_{j=0}^p\binom{p}{j}\phi^{(p-j)}_{\mathbf{A},\mathbf{B},a}(y)h^{(j)}_{\mathbf{A},\mathbf{B}}(y)=\Delta^{(k)}_{\mathbf{A},\mathbf{B},a}(y).
\end{equation}
After rearrangement of terms we obtain the following recursive
formula
\begin{equation}
\label{phiy}
\phi^{(p)}_{\mathbf{A},\mathbf{B},a}(y)=\frac{\Delta^{(p)}_{\mathbf{A},\mathbf{B},a}(y)-
\sum_{j=1}^p\binom{p}{j}\phi^{(p-j)}_{\mathbf{A},\mathbf{B},a}(y)h^{(j)}_{\mathbf{A},\mathbf{B}}(y)}{h_{\mathbf{A},\mathbf{B}}(y)}
\end{equation}
which, upon evaluating at $y=0$, gives (\ref{sub1}).

What remains is to evaluate the successive derivatives
$h_{\mathbf{A},\mathbf{B}}^{(j)}(0)$; equivalently, $x^{(j)}_1(0)$
and $x^{(j)}_2(0)$. To this end, we use the relations
\begin{equation}
\label{matrelation}
\begin{split}
x_1(y)+x_2(y)&=\text{tr}\left(\mathbf{A}\right)+y\text{tr}\left(\mathbf{B}\right)\\
x_1(y)x_2(y)& =\left|\mathbf{A}+\mathbf{B}y\right|=|\mathbf{A}|+|\mathbf{A}|\text{tr}\left(\mathbf{BA}^{-1}\right)y+|\mathbf{B}|y^2.
\end{split}
\end{equation}
Evaluating the first derivative of (\ref{matrelation}) with respect to $y$ at $y=0$ gives
\begin{equation}
\begin{split}
x_1^{(1)}(0)+x_2^{(1)}(0)& =\text{tr}(\mathbf{B})\\
x_2(0)x_1^{(1)}(0)+x_1(0)x_2^{(1)}(0) &=|\mathbf{A}|\text{tr}\left(\mathbf{BA}^{-1}\right)
\end{split}
\end{equation}
which upon solving for $x_1^{(1)}(0)$ and $x_2^{(1)}(0)$ gives the
corresponding results in (\ref{x1def}) and (\ref{x2def}) (i.e.,
$j=1$). Taking the second derivative of (\ref{matrelation}),
followed by similar calculations as before, gives the case
corresponding to $j=2$ in (\ref{x1def}) and (\ref{x2def}). The
remaining case, $j\geq 3$, is more challenging. To proceed, let us
take the $j$th derivative of (\ref{matrelation}) for $j \geq 3$ to
obtain
\begin{equation}
\begin{split}
x_1^{(j)}(y)+x_2^{(j)}(y)& =0\\
\sum_{k=0}^j
\binom{j}{k}
x_1^{(j-k)}(y)x_2^{(k)}(y)&=0
\end{split}
\end{equation}
where we have again used the Leibniz's formula to obtain the $j$th
derivative of the product $x_1(y)x_2(y)$. After some rearrangement
of terms followed by evaluating the resultant derivatives at $y=0$
gives
\begin{equation}
\begin{split}
x_1^{(j)}(0)+x_2^{(j)}(0)& =0\\
x_1^{(j)}(0)x_2(0)+x_2^{(j)}(0)x_1(0)&=-\sum_{k=1}^{j-1}
\binom{j}{k}
x_1^{(j-k)}(0)x_2^{(k)}(0).
\end{split}
\end{equation}
These simultaneous equations can easily be solved for $x_1^{(j)}(0)$
and $x_2^{(j)}(0)$ to yield the results in (\ref{x1def}) and
(\ref{x2def}).
\end{proof}
\section{Proof of Lemma \ref{lem:trace}} \label{ap:C}

For $\mathbf{Z}\in\mathcal{H}_m^+$ and $\mathbf{R}\in\mathcal{H}_m$ with rank
one, we have from \cite[Eq. 6.1.20]{Mathai}
\begin{equation}
\label{mat1}
\int_{\mathbf{X}\in\mathcal{H}_m^+}
\text{etr}\left(-\mathbf{ZX}\right)\left|\mathbf{X}\right|^{a-m}C_{\tau}\left(\mathbf{XR}\right)d\mathbf{X}=(a)_t\tilde \Gamma_m(a)
|\mathbf{Z}|^{-a}C_{\tau}\left(\mathbf{RZ}^{-1}\right)
\end{equation}
where $\Re(a)>m-1$ and $\tau$ is a partition of $t$. Following the
proof of \cite[Lemma 7]{Khatri1966}, let us now select $\mathbf{Z}$
such that $\mathbf{Z}=\mathbf{A}+y\mathbf{I}_{m}$, where
$\mathbf{A}\in\mathcal{H}_m^+$ and $y\geq 0$.
Substituting this specific value of $\mathbf{Z}$ into (\ref{mat1}) and
choosing $\mathbf{R}$ such that $\mathbf{R}=\mathbf{r}\mathbf{r}^H$
where $\mathbf{r}\in C^{m\times 1}$, yields
\begin{align}
\label{mat2}
\int_{\mathbf{X}\in\mathcal{H}_m^+}
\text{etr}\left(-\mathbf{AX}-y\mathbf{X}\right)&\left|\mathbf{X}\right|^{a-m}\text{tr}^t\left(\mathbf{XR}\right)d\mathbf{X}=\nonumber\\
& (a)_t\tilde \Gamma_m(a)\left|\mathbf{A}+y\mathbf{I}_m\right|^{-a}\left(\mathbf{r}^H\left(\mathbf{A}+y\mathbf{I}_m\right)^{-1}\mathbf{r}\right)^t.
\end{align}
Moreover, we can expand the term $\text{etr}\left(-y\mathbf{X}\right)$ to obtain
\begin{align}
\label{infexp}
\sum_{k=0}^\infty \frac{(-1)^ky^k}{k!}
\int_{\mathbf{X}\in\mathcal{H}_m^+}
& \text{etr}\left(-\mathbf{AX}\right)\text{tr}^k\left(\mathbf{X}\right)\left|\mathbf{X}\right|^{a-m} \text{tr}^t\left(\mathbf{XR}\right)d\mathbf{X}= \xi (y)
\end{align}
where
\begin{equation}
\xi (y) := (a)_t\tilde
\Gamma_m(a)\left|\mathbf{A}+y\mathbf{I}_m\right|^{-a}\left(\mathbf{r}^H\left(\mathbf{A}+y\mathbf{I}_m\right)^{-1}\mathbf{r}\right)^t
.
\end{equation}
Now, we seek a power series expansion for the real-valued function
$\xi (y)$
around $y=0$. Equating the coefficient of $y$ with that on the
left-hand side of (\ref{infexp}) will then give the desired
expression.

We require $\xi^{(1)}(0)$. To evaluate this, we start with the
eigen-decomposition
$\mathbf{A}=\mathbf{U}\widetilde{\boldsymbol{\Sigma}}\mathbf{U}^H$,
where $\mathbf{U}\in \mathbb{C}^{m\times m}$ is unitary and
$\boldsymbol{\Sigma}=\text{diag}\left(\widetilde \sigma_1,\widetilde
\sigma_2,\ldots,\widetilde \sigma_m\right)$, to obtain
\begin{align}
\label{eigendecA}
\xi (y)
= & (a)_t\tilde \Gamma_m(a)\frac{\left(\displaystyle \sum_{i=1}^m\frac{|h_i|^2}{y+\widetilde \sigma_i}\right)^t}{\displaystyle \prod_{i=1}^m\left(y+\widetilde \sigma_i\right)^a}
\end{align}
 and $\mathbf{U}^H\mathbf{r} =: \mathbf{h}=\left(h_1\;h_2\;\ldots\;h_m\right)^T$. It is not difficult to see that we can
 have a convergent power series if we select $y<\min\left(\widetilde \sigma_1,\widetilde \sigma_2,\ldots,\widetilde \sigma_m\right)$. Finally equating $\xi^{(1)}(0)$ with the coefficient of $y$ on the left-hand side of
(\ref{infexp}) with $\displaystyle \sum_{i=1}^m\frac{|h_i|^2}{\widetilde \sigma_i^n}=\text{tr}\left(\mathbf{r}^H\left(\mathbf{A}^{-1}\right)^n \mathbf{r}\right)=\text{tr}\left(\mathbf{R}\left(\mathbf{A}^{-1}\right)^n\right)$ gives (\ref{theq1}).
 \end{proof}
\section{Proof of Lemma \ref{lem:2by2tracep}} \label{ap:D}
We combine (\ref{infexp}) and (\ref{eigendecA}) for the
case $m=2$ and apply the relation $\sigma_2|h_1|^2+\sigma_1|h_2|^2=|\mathbf{A}|\text{tr}\left(\mathbf{R}\mathbf{A}^{-1}\right)$ to arrive at
\begin{align}
\label{combine}
\sum_{p=0}^\infty \frac{(-1)^py^p}{p!} \int_{\mathbf{X}\in\mathcal{H}_2^+}
\text{etr}\left(-\mathbf{AX}\right)
\text{tr}^p\left(\mathbf{X}\right)\left|\mathbf{X}\right|^{a-2}\text{tr}^t\left(\mathbf{XR}\right)d\mathbf{X}
= \bar \zeta(y)
\end{align}
where
\begin{align}
\label{zetadef} \bar \zeta(y) &=
\mathcal{P}\frac{(y+b)^t}{\left(\frac{y^2}{|\mathbf{A}|}+2\beta
\frac{y}{\sqrt{|\mathbf{A}|}}+1\right)^{a+t}}
\end{align}
with $\mathcal{P}=\displaystyle \frac{(a)_t\tilde
\Gamma_2(a)\text{tr}^t(\mathbf{R})}{|\mathbf{A}|^{t+a}},\;
\beta=\displaystyle
\frac{\text{tr}(\mathbf{A})}{2\sqrt{|\mathbf{A}|}},\;\text{and}\;
b=\displaystyle
\frac{|\mathbf{A}|\text{tr}\left(\mathbf{RA}^{-1}\right)}{\text{tr}(\mathbf{R})}.$

Our objective is to obtain a power series expansion for $\bar
\zeta(y)$ around $y=0$. To this end, we may use the generating
function definition of ultraspherical
polynomials\footnote{Ultraspherical polynomials can be defined
through the generating function as \cite[Eq. 6.4.10]{Askay}
$(1-2xr+r^2)^{-\lambda}=\sum_{n=0}^\infty
\mathcal{C}_n^\lambda(x)r^n$.}  to write
\begin{equation}
\label{ultra}
\left(\frac{y^2}{|\mathbf{A}|}+2\beta \frac{y}{\sqrt{|\mathbf{A}|}}+1\right)^{-(a+t)}=\sum_{n=0}^\infty\frac{\mathcal{C}_{n}^{a+t}\left(-\beta\right)}{|\mathbf{A}|^{\frac{n}{2}}}y^n.
\end{equation}
Now we may use (\ref{ultra}) in (\ref{zetadef}) with binomial theorem to obtain
\begin{equation}
\label{desiredbinom}
\bar \zeta(y)=\mathcal{P}\sum_{n=0}^\infty \sum_{l=0}^t
\binom{t}{l}b^{t-l}\frac{\mathcal{C}_{n}^{a+t}\left(-\beta\right)}{|\mathbf{A}|^{\frac{n}{2}}}y^{n+l}.
\end{equation}
Since the desired general form of the expansion is 
 \begin{equation}
 \label{desired}
 \bar \zeta(y)=\sum_{p=0}^{\infty}\frac{(-1)^p}{p!}\mathcal{A}_py^p,
 \end{equation}
what is left is to determine the coefficient $\mathcal{A}_p$ using
(\ref{desiredbinom}). To this end, we must collect the coefficients
of $y^p$ together. Since (\ref{desiredbinom}) contains a finite
inner summation, we have to consider two cases depending on the
value of $t$; namely $p\leq t$ and $p> t$. When $p\leq t$, the
summation indices are selected from the set $l,n=\{0,1,2,\ldots,p\}$
such that $l+n=p$. In the case of $p>t$, the summation indices are
selected from the sets $l=\{0,1,2,\ldots,t\}$ and
$n=\{p-t,p-t+1,p-t+2,\ldots,p\}$ such that $l+n=p$. Putting these
together, we come up with a new set
 \begin{equation}
 k=\{0,1,2,\ldots,\min(p,t)\},\; n=p-k
 \end{equation}
which embraces both cases. Thus, the coefficient $\mathcal{A}_p$ can
be written as
   \begin{equation}
 \mathcal{A}_p=\mathcal{P}p!\sum_{k=0}^{\min(p,t)}
 (-1)^k\binom{t}{k}b^{t-k}\frac{\mathcal{C}_{p-k}^{a+t}\left(\beta\right)}{|\mathbf{A}|^{\frac{p-k}{2}}}
 \end{equation}
 where we have used the fact that $\mathcal{C}_n^\nu(-z)=(-1)^n\mathcal{C}_n^\nu(z)$.
Using this, equating the coefficient of $y^p$ in (\ref{desired}) and
(\ref{combine}) concludes the proof.
\end{proof}
\section{Proof of Lemma \ref{lem:3by4rank1}} \label{ap:E}
Before proceeding, it is worth mentioning the following relation
\begin{equation}
\label{110zoanlex}
C_{1,1,0}(\mathbf{X})=|\mathbf{X}|\text{tr}\left(\mathbf{X}^{-1}\right)=|\mathbf{X}|C_{1,0,0}\left(\mathbf{X}^{-1}\right)
\end{equation}
where $\mathbf{X}\in\mathcal{H}_3^+$. Also, for $\mathbf{Z}\in\mathcal{H}_3^+$, we have from \cite[Eq.
3.10]{Rathna}
\begin{equation}
\label{110int}
\int_{\mathbf{X}\in\mathcal{H}_3^+}
\text{etr} \left(-\mathbf{Z}\mathbf{X}\right)|\mathbf{X}|
C_{1,0,0}(\mathbf{X}^{-1})d\mathbf{X}
=
\tilde \Gamma_3(4)|\mathbf{Z}|^{-4}C_{1,0,0}(\mathbf{Z}).
\end{equation}
Following the proof of \cite[Lemma 7]{Khatri1966}, let us substitute
$\mathbf{Z}=\mathbf{A}+\mathbf{R}y$, for $y\geq 0$, into
(\ref{110int}) to obtain
\begin{align}
\label{110det} &\int_{\mathbf{X}\in\mathcal{H}_3^+} \text{etr}
\left(-\mathbf{A}\mathbf{X}-\mathbf{RX}y\right)|\mathbf{X}|
C_{1,0,0}(\mathbf{X}^{-1})d\mathbf{X} \nonumber \\
& \hspace*{2cm} =\tilde
\Gamma_3(4)|\mathbf{A}|^{-4}\left|\mathbf{I}_3+\mathbf{A}^{-1}\mathbf{R}y\right|^{-4}
\text{tr}\left(\mathbf{A}+\mathbf{R}y\right).
\end{align}
Since $\mathbf{R}$ is unit rank, $\mathbf{A}^{-1}\mathbf{R}$ is also
unit rank, and therefore (\ref{110det}) can be written as
\begin{align}
\label{110hypo} & \int_{\mathbf{X}\in\mathcal{H}_3^+} \text{etr}
\left(-\mathbf{A}\mathbf{X}-\mathbf{RX}y\right)|\mathbf{X}|
C_{1,0,0}(\mathbf{X}^{-1})d\mathbf{X} \nonumber \\
& \hspace*{2cm} =\tilde
\Gamma_3(4)|\mathbf{A}|^{-4}\text{tr}\left(\mathbf{A}+\mathbf{R}y\right)
{}_1F_0\left(4;-\text{tr}\left(\mathbf{A}^{-1}\mathbf{R}\right)y\right)
\end{align}
where we have used the relation $1/(1+z)^n={}_1F_0(n;-z)$. Now,
since $y$ is arbitrary, we select
$y<1/\text{tr}\left(\mathbf{A}^{-1}\mathbf{R}\right)$ to obtain a
power series expansion for the right-hand side of (\ref{110hypo}) as
\begin{align}
\label{eq:expand}
& \int_{\mathbf{X}\in\mathcal{H}_3^+} \text{etr}
\left(-\mathbf{A}\mathbf{X}-\mathbf{RX}y\right)|\mathbf{X}|
C_{1,0,0}(\mathbf{X}^{-1})d\mathbf{X} \nonumber \\
& \hspace*{2cm} = \tilde
\Gamma_3(4)|\mathbf{A}|^{-4}\text{tr}\left(\mathbf{A}{+}\mathbf{R}y\right)\sum_{t=0}^\infty
\frac{(4)_t}{t!}\text{tr}^t\left(\mathbf{A}^{-1}\mathbf{R}\right)(-y)^t.
\end{align}
Finally, expanding the left side of (\ref{eq:expand}) as a power series of $y$ followed by equating the coefficient of $(-y)^t$ on both sides with some manipulations conclude the proof.
\end{proof}
\section{Proof of Lemma \ref{lem:tracegamma}} \label{ap:F}
We first solve
$$\int_{\mathbf{X}\in\mathcal{H}_2^+}
\text{etr}\left(-\mathbf{A}\mathbf{X}\right)\text{tr}^p\left(\mathbf{BX}\right)\left|\mathbf{X}\right|^{a-2}C_{\tau}(\mathbf{X})d\mathbf{X}
\; .$$ Subsequent application of the basic property $\sum_{\tau}C_\tau (\mathbf{X})=\text{tr}^t(\mathbf{X})$ will then yield the desired result.

Let us begin with the following matrix integral \cite[Eq. 6.1.20
]{Mathai}
\begin{equation}
\label{intzonaldef}
\int_{\mathbf{X}\in\mathcal{H}_2^+}
\text{etr}\left(-\mathbf{Z}\mathbf{X}\right)\left|\mathbf{X}\right|^{a-2}
C_{\tau}(\mathbf{X})d\mathbf{X}=\tilde \Gamma_2(a)[a]_\tau|\mathbf{Z}|^{-a}C_{\tau}(\mathbf{Z}^{-1})
\end{equation}
where $\mathbf{Z}\in\mathcal{H}_2^+$ and $\Re(a)>1$. Selecting
$\mathbf{Z}=\mathbf{A}+\mathbf{B}y$, where $\mathbf{A},\mathbf{B}\in\mathcal{H}_2^+$,
(\ref{intzonaldef}) becomes
\begin{align}
\label{intmatrixsub} &  \int_{\mathbf{X}\in\mathcal{H}_2^+}
\text{etr}\left(-\mathbf{AX}-\mathbf{BX}y\right)\left|\mathbf{X}\right|^{a-2}
C_{\tau}(\mathbf{X})d\mathbf{X} = \zeta(y)
\end{align}
where
\begin{align}
\label{coeffeq} \zeta(y)=\;& \tilde \Gamma_2(a)[a]_\tau
|\mathbf{A}+\mathbf{B}y|^{-a}C_{\tau}(\left(\mathbf{A}+\mathbf{B}y\right)^{-1}).
\end{align}
Since the left-hand side of (\ref{intmatrixsub}) can be expanded as a power series in $y$,
the remaining task is to find a power series expansion for the
right-hand side of (\ref{intmatrixsub}), i.e., $\zeta(y)$, so that
the coefficient of $y^p$ can be compared on both sides.
To this end, we expand the zonal polynomials in (\ref{coeffeq}) using \cite[Eq. 6.1.12]{Mathai}
to obtain
\begin{equation}
\label{zonalcha}
\zeta(y)=\tilde \Gamma_2(a)[a]_\tau \frac{t!(t_1-t_2+1)}{(t_1+1)!t_2!}|\mathbf{A}+\mathbf{B}y|^{-(a+t_1)}\gamma
\end{equation}
where $\gamma=\frac{\mu_1^{t_1-t_2+1}-\mu_2^{t_1-t_2+1}}{\mu_1-\mu_2}$
and $\mu_1,\mu_2$ are the eigenvalues of $\mathbf{A}+\mathbf{B}y$.
At this point, observe that since $\left(t_1,t_2\right)$ is a
partition of $t$, we can write $t_2=t-t_1$, where
$\left\lceil\frac{t}{2}\right\rceil \leq t_1\leq t$. With this
observation and the aid of Lemma \ref{lem:factorize}, we then obtain
\begin{align}
\label{zonaldetexpan}
\gamma=\sum_{i=0}^{\left\lceil\frac{2t_1-t-1}{2}\right\rceil}
(-1)^i4^ie_i^\tau
\text{tr}^{\varepsilon_{t_1,i}}\left(\mathbf{A}+\mathbf{B}y\right)\left|\mathbf{A}+\mathbf{B}y\right|^{-(a+\varepsilon_{t_1})}
\, .
 \end{align}
Next, with the binomial expansion we get,
\begin{align}
\label{binomial}
\zeta(y)=\overline{K}_{t_1}\tilde \Gamma_2(a)
\sum_{i=0}^{\left\lceil\frac{2t_1-t-1}{2}\right\rceil}\sum_{k=0}^{\varepsilon_{t_1,i}}
(-1)^i4^i& e_i^\tau
\binom{\varepsilon_{t_1,i}}{k}
\text{tr}^{\varepsilon_{t_1,i}-k}(\mathbf{A})\text{tr}^{k}(\mathbf{B})\nonumber\\
& \times \left|\mathbf{A}+\mathbf{B}y\right|^{-(a+\varepsilon_{t_1})}y^k.
\end{align}
where $\overline{K}_{t_1} :=
t!\frac{(a)_{t_1}(a-1)_{t_1-1}\left(2t_1-t+1\right)}{\left(t_1+1\right)!\left(t-t_1\right)!}$.

We now aim to obtain a power series expansion for
$\left|\mathbf{A}+\mathbf{B}y\right|^{-(a+\varepsilon_{t_1})}$ in
terms of $y$. To this end, we may express
\begin{align}
\label{detexpansion1}
\left|\mathbf{A}+\mathbf{B}y\right|^{-(a+\varepsilon_{t_1})}& =
 \frac{|\mathbf{A}|^{-(a+\varepsilon_{t_1})}}{\left(1+2\widetilde \beta\sqrt{|\mathbf{A}^{-1}\mathbf{B}|} y+|\mathbf{A}^{-1}\mathbf{B}|y^2\right)^{a+\varepsilon_{t_1}}}\nonumber\\
& =|\mathbf{A}|^{-(a+\varepsilon_{t_1})}
\sum_{n=0}^\infty
|\mathbf{A}^{-1}\mathbf{B}|^{\frac{n}{2}}\mathcal{C}_n^{a+\varepsilon_{t_1}}\left(\beta\right)(-y)^n
\end{align}
where $\widetilde\beta
:=\frac{\text{tr}\left(\mathbf{A}^{-1}\mathbf{B}\right)}{2\sqrt{\left|\mathbf{A}^{-1}\mathbf{B}\right|}}$
with $|y|<\frac{1}{\sqrt{\left|\mathbf{A}^{-1}\mathbf{B}\right|}}$.
Here, to obtain the last equality in (\ref{detexpansion1}), we have
exploited the generating function definition for ultraspherical
polynomials \cite[Eq. 6.4.10]{Askay}.

Incorporating (\ref{detexpansion1}) into (\ref{binomial}) gives
\begin{align}
\zeta(y)=\overline{K}_{t_1}\tilde \Gamma_2(a)
& \sum_{i=0}^{\left\lceil\frac{2t_1-t-1}{2}\right\rceil}\sum_{k=0}^{\varepsilon_{t_1,i}}
(-1)^{i+k}4^ie_i^\tau
\binom{\varepsilon_{t,i}}{k}
\text{tr}^{\varepsilon_{t_1,i}-k}(\mathbf{A})\text{tr}^{k}(\mathbf{B})\nonumber\\
&\times
\left|\mathbf{A}\right|^{-(a+\varepsilon_{t_1})}\sum_{n=0}^\infty
|\mathbf{A}^{-1}\mathbf{B}|^{\frac{n}{2}}\mathcal{C}_n^{a+\varepsilon_{t_1}}\left(\beta\right) (-y)^{n+k}.
\end{align}
Since we are interested in the coefficient of $(-y)^p$, we have to
re-sum the above series to collect all terms having power $(-y)^p$.
A careful inspection of the above equation reveals that we can
select $n=p-k$ and the upper limit of $k$ as
$\min\left(p,\varepsilon_{t_1,i}\right)$. Thus, we have after some
manipulations
\begin{align}
\zeta(y)=\overline{K}_{t_1} \tilde \Gamma_2(a)
\left|\mathbf{A}\right|^{-a}
\sum_{p=0}^\infty
\sum_{i=0}^{\left\lceil\frac{2t_1-t-1}{2}\right\rceil}
\mathcal{B}_{\tau,p,i}\;(-y)^p.
\end{align}
Now, equating the coefficient of $(-y)^p$ with the corresponding
coefficient in (\ref{intmatrixsub}) yields
\begin{align*}
\label{twotraceeq}
\int_{\mathbf{X}\in\mathcal{H}_2^+}
\text{etr}\left(-\mathbf{A}\mathbf{X}\right)\text{tr}^p\left(\mathbf{BX}\right)
 \left|\mathbf{X}\right|^{a-2} C_{\tau}(\mathbf{X})d\mathbf{X}=
\overline{K}_{t_1} p!\tilde \Gamma_2(a)
\left|\mathbf{A}\right|^{-a}\sum_{i=0}^{\left\lceil\frac{2t_1-t-1}{2}\right\rceil}
\mathcal{B}_{\tau,p,i}.
\end{align*}
Finally, using the basic property $\sum_{\tau}C_\tau (\mathbf{X})=\text{tr}^t(\mathbf{X})$ along with the
fact that
$\sum_{\tau}\equiv\sum_{t_1=\left\lceil\frac{t}{2}\right\rceil}^t$
gives the desired result.
\end{proof}








\end{document}